\documentclass[11pt]{amsart}

\usepackage[
paper=a4paper,
text={147mm,230mm},centering
]{geometry}

\iftrue
\makeatletter
\def\@settitle{%
  \vspace*{-20pt}
  \begin{flushleft}%
    \baselineskip14\p@\relax
    \normalfont\bfseries\LARGE
    \@title
  \end{flushleft}%
}
\def\@setauthors{%
  \begingroup
  \def\thanks{\protect\thanks@warning}%
  \trivlist
  \large \@topsep30\p@\relax
  \advance\@topsep by -\baselineskip
  \item\relax
  \author@andify\authors
  \def\\{\protect\linebreak}%
  \authors
  \ifx\@empty\contribs
  \else
    ,\penalty-3 \space \@setcontribs
    \@closetoccontribs
  \fi
  \normalfont
  \endtrivlist
  \endgroup
}
\def\@setabstracta{%
    \ifvoid\abstractbox
  \else
    \skip@25\p@ \advance\skip@-\lastskip
    \advance\skip@-\baselineskip \vskip\skip@
    \box\abstractbox
    \prevdepth\z@ 
    \vskip-10pt
  \fi
}
\renewenvironment{abstract}{%
  \ifx\maketitle\relax
    \ClassWarning{\@classname}{Abstract should precede
      \protect\maketitle\space in AMS document classes; reported}%
  \fi
  \global\setbox\abstractbox=\vtop \bgroup
    \normalfont\small
    \list{}{\labelwidth\z@
      \leftmargin0pc \rightmargin\leftmargin
      \listparindent\normalparindent \itemindent\z@
      \parsep\z@ \@plus\p@
      
    }%
    \item[\hskip\labelsep\bfseries\abstractname.]%
}{%
  \endlist\egroup
  \ifx\@setabstract\relax \@setabstracta \fi
}
\def\section{\@startsection{section}{1}%
  \z@{-1.2\linespacing\@plus-.5\linespacing}{.8\linespacing}%
  {\normalfont\bfseries\large}}
\def\subsection{\@startsection{subsection}{2}%
  \z@{-.8\linespacing\@plus-.3\linespacing}{.3\linespacing\@plus.2\linespacing}%
  {\normalfont\bfseries}}
\def\subsubsection{\@startsection{subsubsection}{3}%
  \z@{.7\linespacing\@plus.1\linespacing}{-1.5ex}%
  {\normalfont\itshape}}
\def\@secnumfont{\bfseries}
\makeatother
\fi 

\usepackage{amssymb}
\usepackage{mathrsfs}
\usepackage[all]{xy}
\usepackage{pinlabel}
\usepackage[textwidth=1.3in,color=yellow]{todonotes}
\usepackage{amsmath, mathtools}
\usepackage{enumitem}
\usepackage{pb-diagram,pb-xy}
\usetikzlibrary{calc}
\usepackage{youngtab}
\usepackage[boxsize=.3 em]{ytableau}

\usepackage{url}
\usepackage{caption}
\usepackage{subcaption}
\usepackage{fancybox}
\usepackage{wrapfig}
\usepackage{color}
\usepackage{multirow, url}
\usepackage{tikz}
\usetikzlibrary{shapes}
\usepackage{xcolor}

\textwidth=\paperwidth \advance\textwidth by-2\oddsidemargin
\advance\oddsidemargin by-1in
\evensidemargin=\oddsidemargin
\calclayout

\theoremstyle{plain}
\newtheorem{theorem}{Theorem}[section]
 
\newtheorem{thmx}{Theorem}

\newtheorem*{setupx}{Setup}

\newtheorem{proposition}[theorem]{Proposition}
\newtheorem{lemma}[theorem]{Lemma}
\newtheorem{corollary}[theorem]{Corollary}

\theoremstyle{definition}

\newtheorem{definition}[theorem]{Definition}

\newtheorem{example}[theorem]{Example}
\newtheorem{assumption}[theorem]{Assumption}

\theoremstyle{remark}
\newtheorem{remark}[theorem]{Remark}

\newcommand{\C}{\mathbb{C}}

\newcommand{\Q}{\mathbb{Q}}

\newcommand{\R}{\mathbb{R}}

\newcommand{\Z}{\mathbb{Z}}

\newcommand{\CP}{\mathbb{C}P}

\newcommand{\git}{\mathbin{/\mkern-6mu/}}

\def\mcal{\mathcal}
\def\frak{\mathfrak}

\numberwithin{equation}{section} \numberwithin{table}{section}

\def\to{\mathchoice{\longrightarrow}{\rightarrow}{\rightarrow}{\rightarrow}}
\makeatletter
\newcommand{\shortxra}[2][]{\ext@arrow 0359\rightarrowfill@{#1}{#2}}
\def\longrightarrowfill@{\arrowfill@\relbar\relbar\longrightarrow}
\newcommand{\longxra}[2][]{\ext@arrow 0359\longrightarrowfill@{#1}{#2}}

\makeatother
\numberwithin{equation}{section}

\begin{document}                                                                          
\title[Holomorphic disks and GIT quotients]{Holomorphic disks and GIT quotients}

\author{Yoosik Kim}
\address{Department of Mathematics and Institute of Mathematical Science, Pusan National University, Busan, Republic of Korea}
\email{yoosik@pusan.ac.kr}
\thanks{The research of Y.  \ Kim was supported by the National Research Foundation of Korea(NRF) grant funded by the Korea government (MSIT) (RS-2025-16069532 and RS-2020-NR049535).}

\begin{abstract}
Let $G$ be a connected compact Lie group and let $\mathbb{G}$ be its complexification. In this paper, we establish a correspondence between the moduli spaces of holomorphic disks bounded by a $G$-invariant Lagrangian submanifold $L \subseteq X$ and those bounded by its quotient $L/G$ in the GIT quotient $X \git \mathbb{G}$. Under suitable positivity and topological assumptions, we derive a computationally effective formula for the disk potential of $L/G$ from that of $L$ via the {semistable disk potential}, which reflects the choice of a level set of a value of the moment map.
\end{abstract}


\maketitle
\setcounter{tocdepth}{1} 
\tableofcontents

\section{Introduction}
\label{secIntroduction}

Let $G$ be a connected compact Lie group and consider a Hamiltonian $G$-action on a symplectic manifold $(X,\omega)$. Meyer and Marsden--Weinstein \cite{MW74, Mey73} introduced the construction of the symplectic quotient $X \git G$ from the $G$-action, known as symplectic reduction. A particularly influential approach originates from Kirwan’s work \cite{Kir84}, which connects the $G$-equivariant cohomology of $X$ to the ordinary cohomology of the quotient. This perspective has led to a broad program comparing $G$-equivariant invariants of $X$ with invariants of $X \git G$. 

In the context of Lagrangian Floer theory and mirror symmetry, substantial effort has been devoted to the construction of equivariant invariants and to understanding with invariants of its quotient, see \cite{DF18, WX18, HLS20, Caz24, Chr25, FS23, LLL23, LS23, Xia23, KLZ24, PT24, CHLL25} for instance. Especially, Teleman and Fukaya \cite{Tel14, Fuk21} proposed the construction of equivariant Floer theory and equivariant Fukaya categories in general setting and formulated the corresponding mirror symmetry. This has led to an active program aimed at relating the Floer theory of $X$ to that of the quotient $X \git G$.

In this paper, \emph{without} using equivariant Floer theory, we investigate the relationship between the symplectic geometry of a smooth algebraic variety $X$ equipped with a Hamiltonian $G$-action and that of its reduction. Instead, our approach is based on geometric invariant theory (GIT) and the Kempf--Ness theorem, which establishes a bridge between GIT quotients and symplectic reductions. Relying on this correspondence, the main purpose of this paper is to compare holomorphic disks in $X$ with boundary on $L$ and those in the quotient $X \git \mathbb{G}$ with those boundary on $L / G$.

A main difficulty in symplectic reduction is that the construction requires restricting to a level set of the moment map and discards a substantial portion of the ambient space $X$. As a result, it is not straightforward to directly compare holomorphic curves before and after reduction. By contrast, when the GIT quotient is geometric, there exists a holomorphic quotient map 
$$
q \colon X^{ss} \to X \git {G}
$$
defined on the semistable locus, whose image is the quotient. Provided that $X^{ss}$ is nonempty, it contains an open dense subset of $X$ and enables a direct comparison between holomorphic disks in $X$ and those in the quotient via the map $q$.

We describe the setup more precisely. Let $\mathbb{G}$ be the complexification of $G$ and suppose that $\mathbb{G}$ acts linearly on a $G$-invariant projective variety $X \subseteq \CP^N$. The induced $G$-action is Hamiltonian with a moment map $\mu_G$. Assume that $G$ acts freely on the level set $\mu_G^{-1}(0)$. By the Kempf--Ness theorem \cite{KN78}, the GIT quotient $X \git \mathbb{G}$ and the symplectic reduction $X \git G = \mu_G^{-1}(0) / G$ are isomorphic. Moreover, the symplectic and complex structure are compatible via the inclusion $\iota \colon \mu_G^{-1}(0) \to X^{ss}$ as shown in \cite{GS82}. Let $L \subseteq \mu_G^{-1}(0)$ be a $G$-invariant Lagrangian submanifold. 

The quotient map $q$ gives rise to the following geometric framework. Suppose that we have a holomorphic principal $\mathbb{G}$-bundle
$$
q \colon P_\mathbb{G} \coloneqq X^{ss}\to M \coloneqq X \git \mathbb{G}.
$$ 
between K\"{a}hler manifolds $P_\mathbb{G}$ and $M$. We denote the corresponding K\"{a}hler manifolds by $(P_{\mathbb{G}}, \omega, J)$ and $(M, \omega_M, J_M)$. Set $P_G \coloneqq \mu^{-1}_G (0)$ and we assume the following setup.

\begin{setupx}
Suppose that $P_\mathbb{G}$ contains a submanifold $P_G $  such that
\begin{enumerate}
\item the restriction $q |_{P_G} \colon P_G \to M$ is a principal $G$-bundle,
\item the principal $\mathbb{G}$-bundle $P_\mathbb{G}$ is isomorphic to the extension of $P_G$ to $\mathbb{G}$, i.e., $P_\mathbb{G} \simeq P_G \times_G \mathbb{G}$,
\item the K\"{a}hler forms satisfy $q^* \omega_{M} = \omega |_{P_G}$. 
\end{enumerate}
\end{setupx}

For a homotopy class $\beta \in \pi_2(X,L)$, we denote by $\mathcal{M}_1(X, L, J, \beta)$ the moduli space of $J$-holomorphic disks in $X$ with boundary on $L$ in $\beta$ together with one boundary marked point $z_0$. Let $\overline{\mathcal{M}}_1(X, L, J, \beta)$  be the Gromov compactification of $\mathcal{M}_1(X, L, J, \beta)$. The disk counting invariant $n_\beta$ is defined by the degree of the evaluation map $\mathrm{ev}_0 \colon \overline{\mathcal{M}}_1(X, L, J, \beta) \to L$ at the marked point $z_0$. Using the morphism $q$, we define
$$
\widehat{\phi} \colon \mcal{M}_1(P_\mathbb{G}, L, J,\beta) \to \mcal{M}_1(M, K, J_M, q_* \beta), \,\, [(\varphi, z_0)] \mapsto [(q \circ \varphi, z_0)]
$$ 
which induces a map  
$$
\phi \colon \mcal{M}_1(P_\mathbb{G},L,J,\beta)/G \to \mcal{M}_1(X \git \mathbb{G}, L / G, q_*J, q_* \beta).
$$

We now state the main result. A more precise formulation is given in Theorem~\ref{theorem_fundacycle}.

\begin{thmx}[Theorem~\ref{theorem_fundacycle}]\label{thmx_A}
Suppose that $\mathcal{M}_1(X, L, J, \beta) = \overline{\mathcal{M}}_1(X, L, J, \beta)$ and that $\beta$ is regular. Then the map
$$
\phi \colon \mcal{M}_1(P_\mathbb{G},L,J,\beta)/G \to \mcal{M}_1(X \git \mathbb{G}, L / G, q_*J, q_* \beta)
$$
is an orientation-preserving homeomorphism, provided that the orientations are suitably chosen. 

In particular, the counting invariants $n_\beta$ and $n_{q_* \beta}$ coincide.
\end{thmx}

The most important part of the proof is the surjectivity of $\phi$. To establish this, we lift a holomorphic disk $\varphi$ from the quotient to the prequotient. This reduces the problem to constructing a holomorphic trivialization of the bundle over the disk $\varphi$ while preserving the boundary condition. To address this issue, motivated by Schultz's work \cite{Sch21}, we apply Donaldson’s result on the solvability of the Hermitian--Yang--Mills equation with Dirichlet boundary conditions.

Using Theorem~\ref{thmx_A}, we now compare the disk potentials of $L \subseteq X$ and $L/G \subseteq X \git \mathbb{G}$. From now on, we assume that $G$ is a torus, $L$ is a torus, and $L / G$ is also a torus. Recall that the \emph{disk potential}, introduced in \cite{ChoOh, FOOOToric1}, is the generating function of disk counting invariants expressed as a formal sum over relative homotopy classes. It is a Laurent polynomial in which each term corresponds to the moduli space of holomorphic disks of Maslov index two and whose coefficient is precisely the corresponding disk counting invariant.  It governs the deformation theory of Lagrangian Floer cohomology and provides SYZ mirrors \cite{FOOObook, AurouxTduality}.

Our comparison proceeds in two steps.
$$
\xymatrix{
(X,L) & (X^{ss}, L) \ar[d]^{q} \ar[l]_{\iota} \\
  & (X \git \mathbb{G}, L / G)
}
$$
The comparison of disk counting invariants after taking the quotient is established in Theorem~\ref{thmx_A}. 

To analyze the effect of removing the unstable locus, we introduce the notions of semistable and unstable disk classes. A disk class $\beta$ is called \emph{semistable} if the image of each member in $\mathcal{M}(X,L,J,\beta)$ is entirely contained in the semistable locus $X^{ss}$. Otherwise, it is called \emph{unstable}. 

Unstable classes may give rise to compactness issues since the moduli space associated with such a class need not remain compact after restricting to the semistable locus. In nice situations, every disk in such a class intersects the unstable locus, so that no holomorphic disks survive after restricting to $X^{ss}$. We refer to such classes as \emph{fully excluded}.

Assuming that every unstable class of Maslov index two is fully excluded, we define the \emph{semistable disk potential} $W^{ss}_L$ by retaining only the Laurent monomials in $W_L$ corresponding to semistable classes. From $W^{ss}_L$, we compute the disk potential of $L / G$. This construction depends on the choice of moment map level (or stability condition), since the semistable locus $X^{ss}$ varies accordingly, and hence so does $W^{ss}_L$. Finally, we remark that under the inclusion $\iota$, the relative homotopy classes may change, in the sense that distinct disk classes in $X$ can become identified in $X^{ss}$. Nevertheless, we show that this identification does not affect the resulting disk potential.

Assume the following setup.
\begin{setupx}
\begin{enumerate}
\item $G$ acts freely on the level set $\mu_G^{-1}(0)$ (and hence on $L$),
\item every homotopy class of Maslov index two is either semistable or fully excluded, 
\item the Lagrangian torus $L$ is positive, $J$ is regular, $\mathrm{ev}_0$ is a submersion, $X \git G$ is symplectically Fano.
\end{enumerate}
\end{setupx}

Recall that $q_* \colon \pi_1(L) \to \pi_1(L / G)$ and consider the kernel of $q_*$. By using the relations determined by $\ker(q_*)$, we eliminate variables in $W_L^{ss}$ and then obtain the disk potential $W_{L/G}$.

We now state the second main result. A more precise formulation is given in Theorem~\ref{theorem_diskGIT}.

\begin{thmx}[Theorem~\ref{theorem_diskGIT}]\label{thmx_diskGIT}
Under the above assumptions, the disk potential of ${L/G}$ is given by
$$
W_{L / G} (\mathbf{z}) = W_L^{ss} |_{\ker(q_*)} (\mathbf{z}).
$$
\end{thmx}

\begin{remark}
Using equivariant Lagrangian correspondences, Lau--Leung--Li \cite{LLL23} developed foundational results relating the equivariant disk potential $W_{G,L}$ of $L$ and the disk potential $W_{L/G}$ of $L/G$ via moment-level Lagrangian correspondences in a general setting. Their result explains the discrepancy between $W_{G,L}$ and $W_{L / G}$ through disk counting associated with the Lagrangian correspondence.
\end{remark}

We emphasize that one of the main advantages of Theorem~\ref{theorem_diskGIT} is its computational effectiveness. For completely integrable systems arising in toric settings or from toric degenerations, one can apply gradient Hamiltonian flows associated with the components of the moment map to analyze stability. By constructing suitable zigzag paths through complex orbits, we can determine which inverse images of facets (or divisors) are contained in the semistable locus. This makes it possible to compute the semistable disk potential explicitly. Furthermore, it is also useful for lifting disk potentials from the base to the total space.

Using the methods developed in this paper together with degeneration (in stages), we derive the disk potential of the quadric hypersurface $\mathcal{Q}^n \subseteq \CP^{n+1}$. It carries a completely integrable system $\Phi_\mathrm{GZ} \colon \mathcal{Q}^n \to \R^n$, called the Gelfand--Zeitlin system, whose image is a simplex $\Delta_n$. This extends \cite{Kim23}, which computes the disk potential of a monotone Lagrangian torus fiber of $\Phi_\mathrm{GZ}$.

\begin{thmx}[Theorem~\ref{thm_quadricshypdisk}]\label{thmx_C}
Let $n \geq 2$ and let $L_n(\mathbf{u})$ be the torus fiber of the completely integrable system $\Phi_\mathrm{GZ}$ over $\mathbf{u} \in \mathrm{Int}(\Delta_n)$ in $\mathcal{Q}^n$. Then the disk potential is
\begin{equation*}
W_{L(\mathbf{u})}(\mathbf{y}) = 
\frac{1}{y_n} T^{n-u_n} + \frac{y_n}{y_{n-1}}  T^{u_n-u_{n-1}}  + \cdots + \frac{y_2}{y_{1}}  T^{u_1 - u_2} + 2y_2  T^{u_2} +  y_1 y_2 T^{u_1 + u_2}.
\end{equation*}
\end{thmx}

The main idea is that, even when a local $G$-action does not extend to a global $G$-action on $\mathcal{Q}^n$, one can deform $\mathcal{Q}^n$ while preserving the relevant counting invariants so that the degenerated variety admits a $G$-action. This allows us to apply the above framework to compare disk counting invariants via GIT quotients. This perspective will play a key role in the study of disk potentials for isotropic flag varieties. We expect that this reduction procedure will provide an effective method for computing disk potential of higher-dimensional isotropic flag varieties by reducing them to lower-dimensional cases.

The organization of this paper is as follows. In Section 2, we review the relevant background on symplectic reduction, GIT quotients, and their relationship. In Section 3, we establish a correspondence between the moduli spaces before and after taking quotients. In Section 4, we analyze disk potentials on $X$, $X^{ss}$, and $X \git \mathbb{G}$. In Section 5, we present several examples illustrating applications of the main theorems. Finally, in Section 6, we compute the disk potential function for quadric hypersurfaces.

\subsection*{Acknowledgement} 
The author would like to express his gratitude to Siu-Cheong Lau and Xiao Zheng. This work grew out of collaborations and enlightening discussions with them. The author also thanks Yongnam Lee for his interest in this work and Leon Li for his explanations of his joint work with Lau and Leung.

\section{Symplectic reductions and GIT quotients}\label{sec_sympredGIT}

The goal of this section is to review stability for symplectic reductions and GIT quotients as well as the Kempf--Ness theorem, and to collect several basic facts on symplectic reductions and GIT quotients. Our main references are \cite{GRS21, Hos, Tom06, Woo09}.

Let $V$ be a finite-dimensional complex vector space and let $G$ be a compact, connected Lie subgroup of $\mathrm{U}(n+1)$ acting on $V$ via a representation $G \to \mathrm{GL}(V)$. Let $\mathbb{G}$ be the complexification of $G$, regarded as a subgroup of $\mathrm{GL}_{n+1}(\C)$, and consider the induced complexified representation $\mathbb{G} \to \mathrm{GL}(V)$. This linear $\mathbb{G}$-action on $V$ induces a $\mathbb{G}$-action on the projective space $\mathbb{P}(V)$.

Via the inclusion $G \hookrightarrow \mathbb{G}$, the $\mathbb{G}$-action induces a Hamiltonian $G$-action on $\mathbb{P}(V)$ with respect to the Fubini--Study form $\omega_\mathrm{FS}$. We normalize $\omega_\mathrm{FS}$ so that it is induced from the standard symplectic form on the sphere of radius $\sqrt{2}$. A corresponding moment map is given by 
$$
\mu_G \colon \mathbb{P}(V) \to \frak{g}^* \quad \langle \mu_G([v]), \xi \rangle = \frac{\langle v, \sqrt{-1}\xi v \rangle}{|v|^2}.
$$
Let $X \subseteq \mathbb{P}(V)$ be a $\mathbb{G}$-invariant projective variety. Then the $G$-action on $\mathbb{P}(V)$ restricts to a Hamiltonian $G$-action on $X$ with respect to the induced K\"{a}hler form $\omega_\mathrm{FS} |_X$. By abuse of notation, we denote the corresponding moment map again by 
\begin{equation}\label{equ_momentmaprest}
\mu_G \coloneqq \mu_G |_X \colon X \to \frak{g}^*.
\end{equation}

We now recall the notions of stability in symplectic and algebraic geometry.

\begin{definition}\label{def_symplecticstability}
An element $x \in X$ is called 
\begin{itemize}
\item \emph{$\mu$-unstable} if $\overline{\mathbb{G}(x)} \cap \mu_G^{-1}(0) = \emptyset$,
\item \emph{$\mu$-semistable} if $\overline{\mathbb{G}(x)} \cap \mu_G^{-1}(0) \neq \emptyset$,
\item \emph{$\mu$-polystable} if ${\mathbb{G}(x)} \cap \mu_G^{-1}(0) \neq\emptyset$,
\item \emph{$\mu$-stable} if $\mathbb{G}(x) \cap \mu_G^{-1}(0) \neq \emptyset$ and the isotropy subgroup $\mathbb{G}_x$ is discrete.
\end{itemize}
\end{definition}

\begin{definition}\label{def_GITstability}
An element $x = [v] \in X \subseteq \mathbb{P}(V)$ is called 
\begin{itemize}
\item \emph{unstable} if $0 \in \overline{\mathbb{G}(v)}$, 
\item \emph{semistable} if $0 \notin \overline{\mathbb{G}(v)}$, 
\item \emph{polystable} if $\mathbb{G}(v) = \overline{\mathbb{G}(v)}$, 
\item \emph{stable} if $\mathbb{G}(v) = \overline{\mathbb{G}(v)}$ and the isotropy subgroup $\mathbb{G}_v$ is discrete.
\end{itemize}
Note that the choice of representative $v$ does not matter since the $\mathbb{G}$-action is linear. 
\end{definition}

By the Kempf–Ness theorem, these two notions of stability coincide.

\begin{theorem}[\cite{KN78}]
Let $\mu_G \colon X \to \frak{g}^*$ be the moment map in~\eqref{equ_momentmaprest} and let $x = [v] \in X$. Then $x$ is unstable (resp.  semistable, polystable, stable) if and only if $x$ is $\mu$-unstable (resp.  $\mu$-semistable, $\mu$-polystable, $\mu$-stable).
\end{theorem}

We denote by $X^{us}$, $X^{ss}, X^{ps}$, and $X^{s}$ the loci of ($\mu$)-unstable, ($\mu$)-semistable, ($\mu$)-polystable, and ($\mu$)-stable points of $X$, respectively. 

In this situation, we consider two quotient constructions, the symplectic reduction and the GIT quotient. Using the Hamiltonian $G$-action on $X$, we construct the symplectic reduction $X \git G \coloneqq \mu_G^{-1}(0) / G$. On the other hand, using the linear $\mathbb{G}$-action, we construct the GIT quotient $X \git \mathbb{G} \coloneqq \mathrm{Proj} \bigl( \bigoplus_{k} H^0(X, \mcal{O}(k))^\mathbb{G} \bigr)$. Moreover, the inclusion of the graded algebra $\bigoplus_{k} H^0(X, \mcal{O}(k))^\mathbb{G} \to \bigoplus_{k} H^0(X, \mcal{O}(k))$ induces a morphism 
$$
q \colon X^{ss} \to X \git \mathbb{G}.
$$

\begin{theorem}[\cite{KN78}]\label{theorem_KNthm}
The inclusion $\iota \colon \mu^{-1}_G(0) \to X^{ss}$ induces a homeomorphism
$$
\underline{\iota} \colon X \git G \to X \git \mathbb{G}.
$$
\end{theorem}

We now impose additional assumptions on the $G$-action on the level set $\mu_{G}^{-1}(0)$ and recall properties that will be used later. In general, Definition~\ref{def_symplecticstability} implies that
$$
X^s \subseteq X^{ps} \subseteq X^{ss}.
$$ 
When the $G$-action on $\mu_{G}^{-1}(0)$ is locally free, these loci coincide.  

\begin{lemma}[Theorems 4.26, 10.3 in \cite{Hos}]\label{lemma_locallyfreecase}
Suppose that the $G$-action on $\mu_G^{-1}(0)$ is locally free. Then the following statements hold.
\begin{enumerate}
\item $X^s = X^{ps} = X^{ss}$,
\item for every $x \in X^{ps}$, the orbit $\mathbb{G} \cdot x$ intersects $\mu_G^{-1}(0)$ in a single $G$-orbit. In particular, 
$$
X^{ps} = \mathbb{G} \cdot \mu_G^{-1}(0),
$$ 
\item the quotient map $q \colon X^{ss} \to X^{ss} / \mathbb{G} \simeq X \git \mathbb{G}$ is a geometric quotient. 
\end{enumerate}
\end{lemma}

We further assume that the $G$-action on $\mu_{G}^{-1}(0)$ is free. We then consider the reduced symplectic form  $\omega_\mathrm{red}$ on $X \git G$ and the induced complex structure $J_\mathrm{red}$ on $X \git \mathbb{G}$. These structures are compatible via the induced map $\underline{\iota}$ and hence $X \git G \simeq X \git \mathbb{G}$ inherits a K\"{a}hler structure.

\begin{theorem}[Theorems 3.5, 4.5 in \cite{GS82}]\label{proposition_freeGGaction}
Suppose that the $G$-action on $\mu_G^{-1}(0)$ is free. Then the following statements hold. 
\begin{enumerate}
\item the $\mathbb{G}$-action on $X^s = X^{ps} = X^{ss}$ is free and
\item the triple $(X \git G, \omega_\mathrm{red}, \underline{\iota}^*J_\mathrm{red})$ is a K\"{a}hler manifold.
\end{enumerate}
\end{theorem}

By abuse of notation, we denote by $\omega_\mathrm{red}$ and $J_\mathrm{red}$ the reduced symplectic form and the induced complex structure on both $X \git G$ and $X \git \mathbb{G}$. By Lemma~\ref{lemma_locallyfreecase} and Theorem~\ref{proposition_freeGGaction}, after identifying the left $G$-action with a right $G$-action by setting $x * g = g^{-1} * x$, the quotient map $q \colon X^{ss} \to X \git \mathbb{G}$ becomes a holomorphic principal $\mathbb{G}$-bundle. We also have a principal $G$-bundle $q_G \colon\mu^{-1}_G(0) \to X \git G$. These two bundles are related through the inclusion and the induced isomorphism in Theorem~\ref{theorem_KNthm}. We have the commutative diagram.
\begin{equation}
\xymatrix{
\mu_G^{-1}(0) \ar[r]^{\iota} \ar[d]_{q_G} & X^{ss} \ar[d]^q\\
X \git G \ar[r]^{\underline{\iota}}& X \git \mathbb{G}
}
\end{equation}
Moreover, the principal $\mathbb{G}$-bundle $q$ is an extension of $q_G$ via the following map
$$
\mu_G^{-1}(0) \times_G \mathbb{G} \simeq X^{ss}, \quad (x, g) \mapsto x * g.
$$

\begin{proposition}\label{proposition_principalGbundle}
The principal $\mathbb{G}$-bundle $q$ is an extension of $q_G$ by the inclusion $G \to \mathbb{G}$.
\end{proposition}

Suppose that $L \subseteq \mu_G^{-1}(0)$ is a $G$-invariant Lagrangian submanifold. Then the $G$-action on $L$ is free and the quotient $L / G \eqqcolon K$ is a smooth Lagrangian submanifold in $X \git G$. Moreover, the restriction 
$$
q_G |_L \colon L \to K
$$
is a principal $G$-bundle. 

For later use, we record the following proposition. By the Cartan decomposition, there is a diffeomorphism $G \times \frak{g} \to \mathbb{G}$ given by $(g, \xi) \mapsto g \exp( \sqrt{-1} \xi)$. This immediately implies the following proposition.

\begin{proposition}\label{prop_relativehomotoypgroup2}
Let $\mathbb{G}$ be the complexification of a compact connected Lie group $G$. Then 
$$
\pi_2(\mathbb{G}, G) = \{0\}.
$$
\end{proposition}

\begin{remark}
In this section, we have focused on projective GIT quotients arising from Hamiltonian group actions on a projective smooth variety $X$. However, the correspondence of holomorphic disks studied in Section~\ref{sec_holodisksprinp} depends only on the structure of a principal $G$-bundle and its extension to a principal $\mathbb{G}$-bundle. Consequently, the correspondence extends to more general GIT settings including affine GIT quotients. For instance, let $\mathfrak{L} \to X$ be a holomorphic line bundle over a variety $X$ equipped with a $\mathbb{G}$-linearization. The associated GIT quotient $X \git_\mathfrak{L} \mathbb{G}$ is defined with respect to this linearization. Suppose that the quotient map $X^{s}(\mathfrak{L}) \to X \git_\mathfrak{L} \mathbb{G}$ is a principal $\mathbb{G}$-bundle as an extension of the principal $G$-bundle defined by the level set of a central element of the moment map. Then the correspondence of holomorphic disks continues to hold. 
\end{remark}

\section{Holomorphic disks and principal bundles}\label{sec_holodisksprinp}

Let $\mathbb{G}$ be a reductive algebraic group over $\C$ and let $G$ denote its compact real form. Suppose that $\mathbb{G}$ acts linearly on a complex vector space $V$. Let $X \subseteq \mathbb{P}(V)$ be a $\mathbb{G}$-invariant projective variety and $L \subseteq \mu^{-1}_G(0)$ a $G$-invariant Lagrangian submanifold of $X$. 

The main goal of Sections~\ref{sec_holodisksprinp} and~\ref{sec_diskpotentialGIT} is to study the relationship between moduli spaces of holomorphic disks in $X$ with boundary on $L$ and those in the GIT quotient $X \git \mathbb{G}$ with boundary on $L / G$. 

Our approach proceeds in two steps as in~\eqref{equation_twostep}. In the present section, we investigate a correspondence between moduli spaces of $(X^{ss}, L)$ and $(X \git \mathbb{G}, L /G)$ via the quotient map $q$. In the next section, we compare moduli spaces of $(X, L)$ and $(X^{ss}, L)$ via the inclusion $\iota$.
\begin{equation}\label{equation_twostep}
\xymatrix{
(X, L) & (X^{ss}, L)  \ar[l]_\iota \ar[d]^q \\\
& (X \git \mathbb{G}, L /G)}
\end{equation}

Motivated by Proposition~\ref{proposition_principalGbundle}, we work in the following slightly more general setting. Let $q \colon P_\mathbb{G} \to M$ be a holomorphic principal $\mathbb{G}$-bundle  between K\"{a}hler manifolds $P_\mathbb{G}$ and $M$. We denote by $\omega$ (resp. $\omega_M$) the K\"{a}hler form on $P_\mathbb{G}$ (resp. $M$) and by $J$ (resp. $J_M$) the complex structure on $P_\mathbb{G}$ (resp. $M$). Throughout this section, we impose the following assumption.

\begin{assumption}\label{assumption_pgpmathbbg}
Suppose that $P_\mathbb{G}$ contains a submanifold $P_G$  such that
\begin{enumerate}
\item the restriction $q |_{P_G} \colon P_G \to M$ is a principal $G$-bundle,
\item the principal $\mathbb{G}$-bundle $P_\mathbb{G}$ is isomorphic to the extension of $P_G$ to $\mathbb{G}$, i.e., $P_\mathbb{G} \simeq P_G \times_G \mathbb{G}$,
\item the K\"{a}hler forms satisfy $q^* \omega_{M} = \omega |_{P_G}$. 
\end{enumerate}
\end{assumption}

Let $L \, (\subseteq P_G)$ be a $G$-invariant oriented Lagrangian submanifold of $P_\mathbb{G}$. Then the restriction 
$$
q |_{L} \colon L \to L / G \eqqcolon K
$$ 
is a principal $G$-bundle. For simplicity, we denote both $q |_L$ and $q |_{P_G}$ again by $q$. Note that, in the situation of~\eqref{equation_twostep}, $P_\mathbb{G}$, $P_G$, and $M$ correspond to $X^{ss}$, $\mu_G^{-1}(0)$, and $X \git \mathbb{G}$, respectively. 

In this setting, we compare moduli spaces of holomorphic disks in $M$ with boundary on $K$ and those in $P_\mathbb{G}$ with boundary on $L$. 

We begin with preliminary lemmas. The following lemma is also proven in \cite[Proposition 3.2]{Smi21}. We include the proof for the reader's convenience.

\begin{lemma}\label{lemma_isoonhomotopy}
The map $q \colon P_\mathbb{G} \to M$ induces an isomorphism $q_* \colon \pi_2 (P_\mathbb{G}, L) \simeq \pi_2(M, K)$.
\end{lemma}

\begin{proof}
Consider the long exact sequence of homotopy groups for the principal $G$-bundle $q \colon L \to K$. Fix a point $p \in L$. We also consider the long exact sequence of homotopy groups associated to the pair $(L, G \cdot p)$. By aligning up these long exact sequences, we obtain the following commutative diagram
\begin{equation*}
\xymatrix{
\pi_{j+1}(G \cdot p) \ar[r] \ar[d]^{\simeq} & \pi_{j+1}(L) \ar[r] \ar[d]^{\simeq}  & \pi_{j+1}(L, G \cdot p) \ar[r] \ar[d]_{q_*}  & \pi_j(G \cdot p)  \ar[r] \ar[d]^{\simeq}  & \pi_j(L) \ar[d]^{\simeq}\\
\pi_{j+1}(G) \ar[r] & \pi_{j+1}(L) \ar[r] &  \pi_{j+1}(K) \ar[r]  & \pi_j (G) \ar[r] &  \pi_j (L).
}
\end{equation*}
By the five lemma, the induced homomorphism $q_* \colon \pi_{j+1}(L, G \cdot p) \simeq \pi_{j+1}(K)$ is an isomorphism. By applying the same argument to the principal $\mathbb{G}$-bundle $q \colon P_\mathbb{G} \to M$, we check that $q_* \colon \pi_{j+1}(P_\mathbb{G}, \mathbb{G} \cdot p) \simeq \pi_{j+1}(M)$ is an isomorphism. 

Using the above isomorphisms, we obtain the following commutative diagram
\begin{equation}\label{equ_fivelemma}
\xymatrix{
\pi_2(L, G \cdot p) \ar[r] \ar[d]^{\simeq}_{q_*} & \pi_2(P_\mathbb{G}, \mathbb{G} \cdot p) \ar[r] \ar[d]^{\simeq}_{q_*}  & \pi_2(P_\mathbb{G}, L) \ar[r] \ar[d]_{q_*}  & \pi_1(L, G \cdot p)  \ar[r] \ar[d]^{\simeq}_{q_*}  & \pi_1(P_\mathbb{G},G\cdot p) \ar[d]^{\simeq}_{q_*}\\
\pi_2(K) \ar[r] & \pi_2(M) \ar[r] &  \pi_2(M, K) \ar[r]  & \pi_1 (K) \ar[r] &  \pi_1 (M)
}
\end{equation}
Again by the five lemma, we conclude that the induced map $q_* \colon \pi_2 (P_\mathbb{G}, L) \simeq \pi_2(M, K)$ is an isomorphism. 
\end{proof}

\begin{lemma}\label{lemma_qregular}
Let $\beta \in \pi_2(P_{\mathbb{G}},L)$ be a relative homotopy class. Then the following hold.\begin{enumerate}
\item
the homotopy class $\beta$ is regular if and only if  $q_* \beta$ is regular.
\item
the Maslov index of $\beta$ is equal to that of $q_* \beta$.
\end{enumerate}
\end{lemma}

\begin{proof}
Let $VP_\mathbb{G}$ (resp. $VP_G$) be the vertical subbundle of $TP_\mathbb{G}$ (resp. $TP_{G}$). Recall that each fiber of $VP_\mathbb{G}$ (resp. $VP_G$) is generated by the fundamental vector fields of the $\mathbb{G}$-action (resp. ${G}$-action). 
We obtain a short exact sequence of complex vector bundles over $P_\mathbb{G} \colon$
\begin{equation}\label{equ_bundleses}
0 \to VP_{\mathbb{G}} \to TP_{\mathbb{G}} \to q^* TM \to 0
\end{equation}
and a short exact sequence of real subbundles along $L \colon$
\begin{equation}\label{equ_realbundleses}
0 \to VL \to TL \to q^* TK \to 0.
\end{equation}

Take a smooth map $\varphi \colon (\mathbb{D}, \partial \mathbb{D}) \to (P_\mathbb{G}, L)$ representing $\beta$. Then $q \circ \varphi$ represents $q_* \beta$. Pulling back~\eqref{equ_bundleses} and~\eqref{equ_realbundleses} along $\varphi$, we have a short exact sequence of bundle pairs over $(\mathbb{D}, \partial \mathbb{D})$
\begin{equation}\label{equ_sesofbundlepairs}
\xymatrix{
0 \ar[r] & (VP_{\mathbb{G}}, VL) \ar^{\iota}[r] & (TP_{\mathbb{G}}, TL) \ar[r] &  (q^* TM, q^*TK) \ar[r] & 0.
}
\end{equation}
It leads to the following Maslov index formula$\colon$
$$
\mu(\varphi^* TP_\mathbb{G}, \varphi^* TL) = \mu(\varphi^* VP_\mathbb{G}, \varphi^* VL) + \mu((q \circ \varphi)^*TM, (q \circ \varphi)^*TK).
$$
The vertical bundle $VP_\mathbb{G}$ is trivial as a complex vector bundle via the fundamental vector fields, that is, $VP_\mathbb{G} \simeq P_\mathbb{G} \times \frak{g}_{\C}$. Moreover, this trivialization restricts to one of $VP_G \simeq P_G \times \frak{g}$ and $VL \simeq L \times \frak{g}$. Consequently, the bundle pair $(\varphi^* VP_\mathbb{G}, \varphi^* VL)$ is trivial and therefore 
$$
\mu(\varphi^* VP_\mathbb{G}, \varphi^* VL) = 0.
$$
This proves $(2)$.

We now prove~$(1)$. Recall the sheaf-theoretic criterion for regularity from \cite[Section 3.4]{KL01}. Let $(E,F)$ be a Riemann--Hilbert bundle over $(\mathbb{D}, \partial \mathbb{D})$. We denote by $\mcal{A}^0 (E,F)$ the sheaf of local $C^\infty$-sections of $E$ with the boundary condition in $F$ and by $\mcal{A}^{0,1}(E)$ the sheaf of local $C^\infty$ $E$-valued $(0,1)$-forms. These form a fine resolution of the sheaf $(\mathcal{E}, \mathcal{F})$ of local holomorphic sections of $E$ with boundary values in $F$
$$
\xymatrix{	
0 \ar[r] & (\mathcal{E}, \mathcal{F}) \ar[r] & \mcal{A}^0(E, F) \ar^{\overline{\partial}}[r] & \mcal{A}^{0,1}(E) \ar[r] & 0.
}
$$
Then the sheaf cohomology of $(\mathcal{E}, \mathcal{F})$ is given by the twisted Dolbeault complex 
$$
\xymatrix{
0 \ar[r] & {A}^0(E, F) \ar^{\overline{\partial}}[r] & {A}^{0,1}(E) \ar[r] & 0
}
$$
where $A^0(E,F)$ is the space of global $C^\infty$ sections of $E$ with boundary values in $F$ and $A^{0,1}(E)$ is the space of global $C^\infty$ $E$-valued $(0,1)$-forms. In particular,
\begin{itemize}
\item
$H^0(E,F)$ is the kernel of $\overline{\partial}$.
\item
$H^1(E,F) $ is the cokernel of $\overline{\partial}$.
\end{itemize}
Thus, $(E, F)$ is regular if and only if $H^1(E, F)= 0$. 

Applying this to the short exact sequence~\eqref{equ_sesofbundlepairs}, we obtain the associated long exact sequence in cohomology
\begin{align*}
0 &\to H^0((q \circ \varphi)^*TM, (q \circ \varphi)^*TL) \to H^0(\varphi^* TP_\mathbb{G}, \varphi^* TL) \to H^0(\varphi^* VP_\mathbb{G}, \varphi^* VL) \\
&\to H^1((q \circ \varphi)^*TM, (q \circ \varphi)^*TL) \to H^1(\varphi^* TP_\mathbb{G}, \varphi^* TL) \to H^1(\varphi^* VP_\mathbb{G}, \varphi^* VL) \to 0
\end{align*}
Since $\iota^* \colon H^0(\varphi^* TP_\mathbb{G}, \varphi^* TL) \to H^0(\varphi^* VP_\mathbb{G}, \varphi^* VL)$ is surjective, we have a short exact sequence
$$
0 \to H^1((q \circ \varphi)^*TM, (q \circ \varphi)^*TL) \to H^1(\varphi^* TP_\mathbb{G}, \varphi^* TL) \to H^1(\varphi^* VP_\mathbb{G}, \varphi^* VL) \to 0
$$
Thus, it suffices to prove that $H^1(\varphi^* VP_\mathbb{G}, \varphi^* VL) = 0$. A holomorphic section of $(\varphi^* VP_\mathbb{G}, \varphi^* VL)$ can be identified with a holomorphic map
$$
\eta \colon (\mathbb{D}, \partial \mathbb{D}) \to (\frak{g}_\C, \frak{g}) \simeq (\C^N, \R^N).
$$
Such holomorphic disks must be constant. Thus, by \cite[Theorem II]{Oh95}, all partial indices are equal to $0$ and hence $(\varphi^* VP_\mathbb{G}, \varphi^* VL)$ is regular.
Thus, $H^1(\varphi^* VP_\mathbb{G}, \varphi^* VL) = 0$. 
\end{proof}

We fix some notation concerning moduli spaces of holomorphic disks. Let $(X,\omega)$ be a symplectic manifold with an almost complex structure $J$ and $L$ a Lagrangian submanifold. We denote by $j_0$ the standard complex structure on $\mathbb{D} \subseteq \C$. For a relative homotopy class $\beta \in \pi_2(X, L)$, we define the moduli space of $J$-holomorphic disks in $\beta$ by
\begin{itemize}
\item $\widetilde{\mcal{M}}(X,L,J,\beta) = \left\{ \varphi \colon (\mathbb{D}, \partial \mathbb{D}) \to (X,L) \mid d \varphi \circ j_0 = J \circ d \varphi, \, [\varphi(\mathbb{D})] = \beta \right\}$,
\item ${\mcal{M}}(X,L,J,\beta) = \widetilde{\mcal{M}}(X,L,J,\beta) / \mathrm{Aut}(\mathbb{D})$. 
\end{itemize}
We also consider the moduli space of such holomorphic disks with one boundary marked point$\colon$
\begin{itemize}
\item
$
\widetilde{\mcal{M}}_1(X,L,J,\beta) = \left\{ \left( \varphi \colon (\mathbb{D}, \partial \mathbb{D}) \to (X,L) , z_0 \in \partial \mathbb{D} \right) \mid d \varphi \circ j_0 = J \circ d \varphi, \, [\varphi(\mathbb{D})] = \beta \right\}, 
$
\item
$
{\mcal{M}}_1(X,L,J,\beta) = \widetilde{\mcal{M}}_1(X,L,J,\beta) / \mathrm{Aut}(\mathbb{D}),
$
\item 
$
\overline{\mcal{M}}_1(X,L,J,\beta) 
$
is the Gromov compactification of ${\mcal{M}}_1(X,L,J,\beta)$.
\end{itemize}

Suppose that the class $\beta$ is regular and $\mcal{M}_1(X,L,J,\beta) = \overline{\mcal{M}}_1(X,L,J,\beta)$. Then this moduli space is a compact manifold without boundary. Fix orientations on both $\mathcal{M}_1(X,L,J,\beta)$ and $L$.

\begin{definition}\label{def_diskcountinginv}
The \emph{disk counting invariant} (or \emph{open Gromov--Witten invariant}) $n_\beta$ associated to $\beta$ is defined by the degree of the evaluation map
\begin{equation}\label{equ_degreeev}
\mathrm{ev}_0 \colon \mcal{M}_1(X,L,J,\beta) \to L, \quad [(\varphi, z_0)] \mapsto \varphi(z_0).
\end{equation}
Note that the invariant can be nonzero only when of $\dim \mcal{M}_1(X,L,J,\beta) = \dim L$. 
\end{definition}

We now return to the setting introduced earlier. Since the $G$-action on $P_\mathbb{G}$ preserves both the holomorphic structure and the boundary condition, it induces an action on $\mcal{M}_1(P_\mathbb{G},L,J,\beta)$ defined by
\begin{equation}\label{equ_Gactiononmoduli2}
\mcal{M}_1(P_\mathbb{G},L,J,\beta) \times G \to \mcal{M}_1(P_\mathbb{G},L,J,\beta), \quad ([(\varphi, z_0)], g) = [(\varphi * g, z_0)]
\end{equation}
 The following proposition claims that the induced action is free.

\begin{proposition}\label{proposition_freenessonmoduli}
The $G$-action on $\mcal{M}_1(P_\mathbb{G},L,J,\beta)$ defined in~\eqref{equ_Gactiononmoduli2} is free. In particular, if $\beta \in \pi_2(P_\mathbb{G}, L)$ is regular, then the quotient map 
$$
q_* \colon \mcal{M}_1(P_\mathbb{G},L,J,\beta) \to \mcal{M}_1(P_\mathbb{G},L,J,\beta)/G
$$ 
is a principal $G$-bundle.
\end{proposition}

\begin{proof}
To show that the action~\eqref{equ_Gactiononmoduli2} is free, for $[(\varphi, z_0)] \in \mcal{M}_1(P_\mathbb{G},L,J,\beta)$, suppose that \begin{equation}\label{equ_gactionacts}
[(\varphi , z_0)] * g = [(\varphi, z_0)].
\end{equation}
Since the reparametrization $\mathrm{PSL}(2,\R)$-action preserves the image of $\varphi$ at the marked point $z_0$,~\eqref{equ_gactionacts} yields $\varphi(z_0) * g = \varphi(z_0) \in L$. Since $G$ acts freely on $L$, it follows that $g = \mathrm{id}$. Hence the induced $G$-action on $\mcal{M}_1(P_\mathbb{G},L,J,\beta)$ is free.
\end{proof}

\begin{remark}
The presence of the boundary marked point is essential in Proposition~\ref{proposition_freenessonmoduli}. Without a marked point, a nontrivial element of $G$ may act trivially on a holomorphic disk after composing with a nontrivial reparametrization of the domain, producing a nontrivial stabilizer subgroup.
\end{remark}

To relate holomorphic disks in the total space and those in the quotient, consider the map
\begin{equation}\label{equ_mapphidiff}
\widehat{\phi} \colon \mcal{M}_1(P_\mathbb{G}, L, J, \beta)  \to \mcal{M}_1(M, K, J_M, q_* \beta) \quad [(\varphi, z_0)] \mapsto [(q \circ \varphi, z_0)].
\end{equation}
The map $\widehat{\phi}$ is well-defined since the projection $q$ is holomorphic and satisfies $q(L) = K$. Since $P_\mathbb{G}$ is the extension of $P_G$, any two points lying in the same $\mathbb{G}$-orbit have the same image under $q \colon P_\mathbb{G} \to M$. It follows that $\widehat{\phi}$ is $G$-invariant. Consequently, the map $\widehat{\phi}$ factors through the quotient $\mcal{M}_1(P_\mathbb{G},L,J,\beta)/G$ and induces a map $\phi$ on the quotient. Thus we obtain the following commutative diagram.
\begin{equation}\label{equ_mapphidia}
\xymatrix{
\mcal{M}_1(P_\mathbb{G},L,J,\beta) \ar^{\widehat{\phi} }[rr] \ar_{q_*}[rd] & & \mcal{M}_1(M, K, J_M, q_* \beta)\\
 & \mcal{M}_1(P_\mathbb{G},L,J,\beta)/G \ar_{\phi}[ru] &
 }
\end{equation}

We claim that the induced map $\phi$ is an orientation-preserving homeomorphism.  To prove surjectivity of $\phi$, we introduce the following notion.
\begin{definition}
Let $\underline{\varphi} \colon (\mathbb{D},\partial \mathbb{D}) \to (M,K)$ be a $J_M$-holomorphic disk.
A map $\varphi \colon (\mathbb{D}, \partial \mathbb{D}) \to (P_\mathbb{G}, L)$ is called a \emph{holomorphic lift} of $\underline{\varphi}$ if $\varphi$ is $J$-holomorphic and satisfies $q \circ \varphi = \underline{\varphi}$. Equivalently, the following diagram commutes.
$$
\xymatrix{
& (P_\mathbb{G}, L) \ar[d]^q \\\
 (\mathbb{D}, \partial \mathbb{D}) \ar@{-->}^{\varphi}[ru] \ar_{\underline{\varphi}}[r] & (M, K)}
$$
\end{definition}

We now study the existence of a holomorphic lift of a given $J_M$-holomorphic disk $\underline{\varphi}.$ 

\begin{proposition}\label{proposition_existenceoflift}
Every $J_M$-holomorphic disk $\underline{\varphi} \colon (\mathbb{D}, \partial \mathbb{D}) \to (M, K)$ admits a holomorphic lift. 
\end{proposition}

\begin{proof}
Let $Q_\mathbb{G} \coloneqq \underline{\varphi}^* P_{\mathbb{G}}$ be the pullback of the holomorphic principal $\mathbb{G}$-bundle $q \colon P_\mathbb{G} \to M$. 
$$
\xymatrix{
Q_\mathbb{G} \coloneqq \underline{\varphi}^* P_\mathbb{G} \ar[d] \ar[r] & P_\mathbb{G}\ar[d]^q \\\
\mathbb{D} \ar[r]_{\underline{\varphi}} & M}
$$
A holomorphic lift of $\underline{\varphi}$ is equivalent to a holomorphic section of $Q_\mathbb{G}$ whose boundary values lie in the subbundle $Q_G \coloneqq  \underline{\varphi}^* L \, ( \subseteq Q_\mathbb{G} |_{\partial \mathbb{D}})$ over $\partial \mathbb{D}$. Note that $Q_G =  \underline{\varphi}^* L \subseteq Q_\mathbb{G} |_{\partial \mathbb{D}}$ is a reduction of the principal $\mathbb{G}$-bundle $Q_\mathbb{G} |_{\partial \mathbb{D}}$ to $G$. 

Since $Q_\mathbb{G} \to \mathbb{D}$ is holomorphically trivial, one obtains a holomorphic section of $Q_\mathbb{G}$ by choosing a constant section of its holomorphic trivialization. However, such a trivialization need not respect the reduction $Q_G$. More precisely, the image of $\partial \mathbb{D} \times G$ under the holomorphic trivialization may not be contained in $Q_G$. 

To construct a holomorphic trivialization compatible with the boundary condition, we employ Donaldson's result on the Dirichlet problem of Hermitian--Yang--Mills (HYM) equation.

Recall that $\mathbb{G}$ embeds into $\mathrm{GL}_{n+1}(\C)$. Consider the linear $\mathbb{G}$-action on a complex vector space $V$. Let $\rho$ be the standard representation induced by the linear $\mathbb{G}$-action and consider the associated holomorphic vector bundle $E \coloneqq Q_\mathbb{G} \times_\rho V$ over $\mathbb{D}$. Its $\mathbb{G}$-frame bundle $\mathrm{Fr}_\mathbb{G}(E)$ is canonically isomorphic to $Q_\mathbb{G}$. 

Choose a Hermitian metric $h^\prime$ on the bundle $E$ such that $\mathrm{Fr}_{G}(E) |_{\partial \mathbb{D}} = Q_G$ over $\partial \mathbb{D}$ under the above identification $Q_\mathbb{G} \simeq \mathrm{Fr}_\mathbb{G}(E)$. In general, the metric $h^\prime$ need not be flat. To obtain a flat metric preserving the prescribed boundary condition, we use the following theorem, which is the two-dimensional case of \cite[Theorem~1]{Don92}. It enables us to deform $h^\prime$ into a Hermitian metric $h$ satisfying HYM equation while keeping its boundary values fixed. Since $\dim \mathbb{D} = 2$, the HYM metric is merely flat. 
 
\begin{theorem}[cf. Theorem 1 in \cite{Don92}]\label{theorem_Don}
Let $E$ be a holomorphic vector bundle over a Riemann surface $\Sigma$. For any Hermitian metric $h^\prime$ on the restriction of $E$ to $\partial \Sigma$, there exists a unique \emph{flat} Hermitian metric $h$ on $\Sigma$ such that $h = h^\prime$ over $\partial \Sigma$.
\end{theorem}

Applying Theorem~\ref{theorem_Don}, we obtain a flat Hermitian metric $h$ on $E$ extending the boundary metric $h^\prime$. The associated Chern connection associated to $(E, h)$ is flat. Since $\mathbb{D}$ is simply connected, the flat connection has trivial monodromy and its parallel transport yields a holomorphic trivialization of $E$ (and hence one of $\mathrm{Fr}_\mathbb{G}(E)$). Moreover, since the connection preserves the Hermitian metric, it also preserves the unitary frames. Consequently, the trivialization maps $\partial \mathbb{D} \times G$ to $Q_G$. We therefore obtain the following commutative diagram.
\begin{equation}\label{equ_commutativitiyofcompletelyintelift}
\xymatrix{
(\mathbb{D} \times \mathbb{G}, \partial \mathbb{D} \times  {G}) \ar[r] \ar[rd]^{\mathrm{pr}_1} & ( Q_\mathbb{G}, Q_G) \ar[r] \ar[d] & (P_\mathbb{G}, L) \ar[d]^q \\\
& (\mathbb{D}, \partial \mathbb{D}) \ar@{-->}^{\varphi}[ru]\ar@/^1.0pc/^{s}[lu] \ar_{\underline{\varphi}}[r] & (M, K)   }
\end{equation} 
Taking a constant section $s = (id, g)$ for some fixed $g\in G$ and composing it with the holomorphic maps in the top row, we obtain a holomorphic lift $\varphi$ of $\underline{\varphi}$.
\end{proof}

\begin{remark}
The idea of using Donaldson's work originates in \cite{Sch21}, where a similar construction was considered in a slightly different setting.
\end{remark}

As an immediate consequence of Proposition~\ref{proposition_existenceoflift}, the map $\widehat{\phi}$ is surjective and hence so is $\phi$.

\begin{corollary}\label{corollary_ontophi}
The map $\phi \colon \mcal{M}_1(P_\mathbb{G},L,J,\beta)/G \to \mcal{M}_1(M, K, J_M, q_* \beta)$ is surjective.
\end{corollary}

We next prove that the map $\phi$ is injective. 

\begin{lemma}\label{lemma_onetoonephi}
The map $\phi \colon \mcal{M}_1(P_\mathbb{G},L,J,\beta)/G \to \mcal{M}_1(M, K, J_M, q_* \beta)$ is injective.
\end{lemma}

\begin{proof}
Suppose that 
$$
 [(q \circ \varphi, z_0)]) = \phi([(\varphi, z_0)]) = \phi([(\varphi^\prime, z_0^\prime)]) = [(q \circ \varphi^\prime, z^\prime_0)].
$$ 
After replacing $[(\varphi^\prime, z_0^\prime)]$ by an equivalent representative if necessary, we may assume that
$$
\underline{\varphi} \coloneqq q \circ \varphi = q \circ  \varphi^\prime \mbox{ and } z_0 = z_0^\prime.
$$
Thus $\varphi$ and $\varphi^\prime$ are two holomorphic lifts of $\underline{\varphi}$. 

Consider the pullback principal $\mathbb{G}$-bundle $Q_\mathbb{G} \coloneqq \underline{\varphi}^* P_\mathbb{G}$ together with the reduction $Q_{G} \coloneqq \underline{\varphi}^* L \subseteq Q_\mathbb{G} |_{\partial \mathbb{D}}$ over $\partial \mathbb{D}$. The holomorphic lift $\varphi$ determines a holomorphic section of $Q_\mathbb{G} \to \mathbb{D}$ whose boundary values lie in $Q_G$. Such a section yields a holomorphic trivialization 
$$
(\mathbb{D} \times \mathbb{G}, \partial \mathbb{D} \times G) \to (Q_\mathbb{G}, Q_G).
$$
Under this trivialization, the map $\varphi$ corresponds to a holomorphic section $s_\varphi$ of $ \mathrm{pr}_1 \colon \mathbb{D} \times \mathbb{G} \to \mathbb{D}$. We obtain the following commutative diagram.
\begin{equation}\label{equ_commutativitiyofcompletelyinte}
\xymatrix{
(\mathbb{D} \times \mathbb{G}, \partial \mathbb{D} \times G) \ar[r] \ar[rd]^{\mathrm{pr}_1} & (Q_\mathbb{G}, Q_G)  \ar[r] \ar[d] & (P_\mathbb{G}, L) \ar[d]^q \\\
& (\mathbb{D}, \partial \mathbb{D}) \ar[ru]^\varphi \ar@{-->}@/^1.0pc/^{s_\varphi}[lu] \ar_{\underline{\varphi}}[r] & (M, K)   }
\end{equation} 

Using the fixed trivialization determined by $\varphi$, $\varphi^\prime$ determines another holomorphic section $s_{\varphi^\prime}$ of $\mathrm{pr}_1$. Since $\pi_2 (\mathbb{G},G)$ is trivial by Proposition~\ref{prop_relativehomotoypgroup2}, every holomorphic section of $\mathrm{pr}_1$ with boundary values in $\partial \mathbb{D} \times G$ must be constant. Hence both $s_\varphi$ and $s_{\varphi^\prime}$ are constant sections. Therefore, there exists an element $g\in G$ such that 
$$
s_{\varphi'} = s_\varphi * g.
$$
It implies that $\varphi * g = \varphi^\prime$ and hence $[(\varphi, z_0)] = [(\varphi^\prime, z^\prime_0)]$ in $\mathcal{M}_1 (P_\mathbb{G}, L, J, \beta)/G$.
\end{proof}

We now discuss the orientations on the moduli spaces in the domain and codomain of the map $\phi$. Since the action of the connected Lie group $G$ on the oriented $G$-invariant submanifold $L$ preserves orientation, the quotient $K$ inherits a natural orientation under the convention 
\begin{equation}\label{equ_orientreduc}
T_\bullet L \simeq T_\bullet K \times \frak{g}.
\end{equation}  

To orient the moduli space $\mcal{M}_1(M, K, J_M, q_* \beta)$, we assume that $K$ is spin and fix a spin structure on $TK$. Since the vertical subbundle $VL$ of $TL$ is trivial, its second Stiefel--Whitney class $w_2(VL)$ vanishes. By choosing a $G$-invariant connection, we obtain a splitting 
\begin{equation}\label{equ_splittingG}
TL \simeq q^* TK  \oplus VL 
\end{equation}
and
$$
w_2(TL) = w_2( q^* TK ) +  w_2(VL) = 0.
$$
Therefore, $L$ is also spin, see \cite[Proposition II.1.15]{LM86}. 

We equip $q^* TK$ with the pullback spin structure induced from the chosen spin structure on $TK$. Since $VL$ is a trivial vector bundle of rank $r$, its $\mathrm{SO}(r)$-frame bundle is trivial. We equip $VL$ with the trivial spin structure induced by the double covering $\mathrm{Spin}(r) \to \mathrm{SO}(r)$. The splitting~\eqref{equ_splittingG} then determines a product spin structure on $TL$ by extending
\begin{equation}\label{equ_choiceofspinstr}
\mathrm{Fr}_{\mathrm{Spin}(r)}(VL) \times q^* \mathrm{Fr}_{\mathrm{Spin}(s)}(TK)
\end{equation}
through the homomorphism $\mathrm{Spin}(r) \times \mathrm{Spin}(s) \to \mathrm{Spin}(r+s = m)$ where $s = \mathrm{rank}(TK)$.

To explain the proof of the main theorem, we briefly recall the orientation for moduli spaces of holomorphic disks in \cite{FOOObook, Cho04}. Fix a CW complex structure of $L$. Since $L$ is orientable, $TL$ admits a trivialization over the $1$-skeleton of $L$. Such a trivialization determines a loop $\ell$ in $\mathrm{SO}(m)$ along the boundary of each $2$-cell. The chosen spin structure on $TL$ lifts $\ell$ to a loop $\widetilde{\ell}$ in $\mathrm{Spin}(m)$. Since $\mathrm{Spin}(m)$ is simply connected, the loop $\widetilde{\ell}$ is null-homotopic and hence the trivialization extends over each $2$-cell. Indeed, a spin structure on $TL$ is equivalent to a homotopy class of a trivialization of $TL$ over the $1$-skeleton of $L$ which can be extended to the $2$-skeleton of $L$, see \cite{Mil63}. 

This trivialization is used to orient the moduli spaces of holomorphic disks with boundary condition $L$ in \cite{FOOObook}. Let $\varphi \in \mcal{M}(P_\mathbb{G},L,J,\beta)$. By pinching a concentric circle of a collar neighborhood of $\partial \mathbb{D}$ to the origin $O$, the orientation problem reduces to orienting the kernel of the following surjective map
$$
\mathcal{M}( (\mathbb{D}, \partial \mathbb{D}), (\varphi^* TP_\mathbb{G}, \varphi^* TL)) \times \mathcal{M}( \mathbb{P}^1, \varphi^* TP_\mathbb{G}) \to (\varphi^* TP_\mathbb{G})_S, \quad (\xi_1, \xi_2) \mapsto \xi_1(O) - \xi_2(S)
$$
where $O$ is the origin of $\mathbb{D}$ and $S$ is the south pole of $\mathbb{P}^1$. The moduli space $\mathcal{M}( \mathbb{P}^1, \varphi^* TP_\mathbb{G}) $ and the codomain carry the canonical complex orientations. The orientation of the kernel is determined by that of the first factor $\mathcal{M}( (\mathbb{D}, \partial \mathbb{D}), (\varphi^* TP_\mathbb{G}, \varphi^* TL))$. This trivialization of $TL$ induced by the spin structure identifies the space with $\mathcal{M}( (\mathbb{D}, \partial \mathbb{D}), (\C^m, \R^m)) \simeq \R^m$ and determines an orientation. 

For another map $\varphi^\prime$, choose a path connecting $\varphi$ and $\varphi^\prime$. The orientation extends along this path, and the resulting orientation is independent of the choice of path because $L$ is spin (see \cite{Cho04}). This construction also induces an orientation  on $\mcal{M}_1(P_\mathbb{G}, L, J, \beta)$ via the conventions
$$
\widetilde{\mcal{M}}_1 \simeq \widetilde{\mcal{M}} \times \partial \mathbb{D} \mbox{ and } {\mcal{M}}_1 \simeq ( \widetilde{\mcal{M}}_1 / \mathrm{PSL}(2,\R) ) \times  \mathrm{PSL}(2,\R).
$$

We are now ready to state the main theorem of this section.

\begin{theorem}[Theorem~\ref{thmx_A}]\label{theorem_fundacycle}
Suppose that the quotient space $K = L / G$ is spin, $\beta \in \pi_2(P_\mathbb{G}, L)$ is regular, and $\mcal{M}_1(P_\mathbb{G},L,J,\beta) = \overline{\mcal{M}}_1(P_\mathbb{G},L,J,\beta)$. Equip $L$ with the product spin structure described above. Then the smooth map~\eqref{equ_mapphidia}
$$
\phi \colon \mcal{M}_1(P_\mathbb{G},L,J,\beta)/G \to \mcal{M}_1(M, K, J_M, q_* \beta)
$$
is an orientation-preserving homeomorphism. 

In particular, the induced homomorphism $\phi_*$ maps the fundamental cycle of $\mcal{M}_1(P_\mathbb{G},L,J,\beta)/G$ to that of $\mcal{M}_1(M, K, J_M, q_* \beta)$.
\end{theorem}

\begin{proof}
By Corollary~\ref{corollary_ontophi} and Lemma~\ref{lemma_onetoonephi}, the map $\phi$ is bijective. By Lemma~\ref{lemma_qregular} and Proposition~\ref{proposition_freenessonmoduli}, both the domain and codomain of $\phi$ are smooth manifolds of the same dimension. Since $\mcal{M}_1(M, K, J_M, q_* \beta)$ is Hausdorff and $\mcal{M}_1(P_\mathbb{G},L,J,\beta)/G$ is compact, it follows that $\phi$ is a homeomorphism.

It remains to verify that $\phi$ preserves orientations. By our choice of the spin structure in~\eqref{equ_choiceofspinstr}, the $\mathrm{Spin}(m)$-frame bundle of $TL$ reduces to a $\mathrm{Spin}(r) \times \mathrm{Spin}(s)$-bundle corresponding to the splitting~\eqref{equ_splittingG}. This induces a product trivialization of $TL$ over the $1$-skeleton of $L$, obtained from the chosen trivialization of $q^*TK$ together with the canonical trivialization of $VL$. The first factor determines the orientation on $\mcal{M}_1(M, K, J_M, q_* \beta)$ and the second factor corresponds to the infinitesimal $G$-action. Consequently, the orientation on $\mcal{M}_1(M,K,J_M,q_* \beta)$ is determined by the convention 
$$
T_\bullet \mcal{M}_1(P_\mathbb{G},L,J,\beta) \simeq T_\bullet \mcal{M}_1(M,K,J_M,q_* \beta) \times \frak{g}.
$$ 
On the other hand, the orientation on $\mcal{M}_1(P_\mathbb{G},L,J,\beta)/G$ is determined by the convention 
$$
T_\bullet \mcal{M}_1(P_\mathbb{G},L,J,\beta) \simeq T_\bullet \mcal{M}_1(P_\mathbb{G},L,J,\beta)/G \times \frak{g}.
$$
Hence the smooth map $\phi$ is orientation preserving. 
\end{proof}

We now compare disk counting invariants. 

\begin{corollary}\label{cor_countinginv}
Assume the hypotheses of Theorem~\ref{theorem_fundacycle}. Then
$$
n_\beta = n_{q_*{\beta}}.
$$
\end{corollary}

\begin{proof}
By Lemma~\ref{lemma_qregular}, the Maslov indices of $\beta$ and $q_* \beta$ coincide. Hence, it suffices to consider the case in which the Maslov index is equal to two. The evaluation map $\mathrm{ev}_0 \colon \mcal{M}_1 (P_\mathbb{G},L,J,\beta) \to L$ is $G$-equivariant. Therefore, it induces a well-defined map
$$
{\mathrm{ev}}_0 \colon \mcal{M}_1 (P_\mathbb{G},L,J,\beta) / G \to L / G = K.
$$
We then have the following commutative diagram.
\begin{equation}\label{equ_commutativitiyofcompletelyinte2}
\xymatrix{
\mcal{M}_1 (P_\mathbb{G},L,J,\beta) \ar^{\quad \quad \quad \mathrm{ev}_0 }[r] \ar_{q}[d] & L \ar^{q}[d]\\
\mcal{M}_1 (P_\mathbb{G},L,J,\beta)/G \ar_{\quad \quad \quad {\mathrm{ev}}_0}[r] & K  }  
\end{equation} 
By the orientation conventions described above, integration $q_*$ along the $G$-fibers sends the fundamental cycles of $\mcal{M}_1(P_\mathbb{G},L,J,\beta)$ and $L$ to the fundamental cycles of their quotients. Consequently, the degree of the top evaluation map agrees with that of the bottom evaluation map in~\eqref{equ_commutativitiyofcompletelyinte2}.

Next, consider the following commutative diagram
\begin{equation}\label{equ_commutativitiyofcompletelyinte3}
\xymatrix{
\mcal{M}_1 (P_\mathbb{G},L,J,\beta)/G \ar^{\quad \quad \quad \, \mathrm{ev}_0 }[r] \ar_{{\phi}}[d] & K \ar[d]^{=}  \\
\mcal{M}_1 (M, K, J_M, q_* \beta)\ar_{\quad \quad \quad \, \mathrm{ev}_0}[r] &  K   }
\end{equation} 
By Theorem~\ref{theorem_fundacycle}, the homeomorphism $\phi$ preserves orientations and maps the fundamental cycle $[\mcal{M}_1 (P_\mathbb{G},L,J,\beta)/G]$ to $[\mcal{M}_1 (M, K, J_M, q_* \beta)]$. Therefore, the degree of the top evaluation map agrees with that of the bottom evaluation map in~\eqref{equ_commutativitiyofcompletelyinte3}.

Combining the above two comparisons, we conclude that $n_\beta = n_{q_* \beta}$.
\end{proof}

\section{Disk potential functions and GIT quotients}\label{sec_diskpotentialGIT}

The aim of this section is to investigate the relationship between the disk potentials of the prequotient and the GIT quotient.

\subsection{Disk potential functions}

We begin by briefly recalling the relevant notions from Lagrangian Floer theory, focusing on the construction of the disk potential function. Let $\Lambda$ be the Novikov field over $\C$ defined by
\begin{equation}\label{eq_Novikovfield}
\Lambda \coloneqq \left\{ \sum_{j=1}^{\infty} a_{j} T^{\lambda_j}  \, \bigg{|} \, a_j \in \C, \lambda_j \in \R, \lim_{j \to \infty} \lambda_j = \infty \right\},
\end{equation}
where $T$ is a formal parameter recording the symplectic area. We also consider the following subrings$\colon$ 
\begin{align*}
\Lambda_0 &\coloneqq \left\{ \sum_{j=1}^{\infty} a_{j} T^{\lambda_j} \in \Lambda  \, \bigg{|} \, a_j \in \C, \lambda_j \in \R, \lambda_j \geq 0, \lim_{j \to \infty} \lambda_j = \infty \right\} \\
\Lambda_+ &\coloneqq \left\{ \sum_{j=1}^{\infty} a_{j} T^{\lambda_j} \in \Lambda  \, \bigg{|} \, a_j \in \C, \lambda_j \in \R, \lambda_j > 0, \lim_{j \to \infty} \lambda_j = \infty \right\}. \,
\end{align*}

Let $(X, \omega)$ be a compact symplectic manifold of dimension $2m$ and let $L$ be a Lagrangian torus. Fix a compatible almost complex structure $J$ of $X$. For a nonnegative integer $k$ and a class $\beta \in \pi_2(X,L)$, we denote by $\mathcal{M}_{k+1}(X,L,J,\beta)$ the moduli space of $J$-holomorphic disks with $(k+1)$ boundary marked points and by $\overline{\mathcal{M}}_{k+1}(X,L,J,\beta)$ the Gromov compactification of $\mathcal{M}_{k+1}(X,L,J,\beta)$. For $j=0,1, \cdots, k$, define the evaluation maps
$$
\mathrm{ev}_j \colon \overline{\mcal{M}}_{k+1}(X, L, J, \beta) \to L, \quad [(\varphi, z_0, \cdots, z_k)] \mapsto \varphi(z_j).
$$
Set $\mathrm{ev}_+ \coloneqq (\mathrm{ev}_1,\cdots,\mathrm{ev}_k)$ and let $\pi_i \colon L^k \to L$ be the projection to the $i$-th factor. 

Applying the smooth correspondence construction of \cite{Fuk10, FOOO20} to
\begin{equation}\label{equ_smoothcorrespondence}
\xymatrix{
&\overline{\mcal{M}}_{k+1}(X, L, J, \beta) \ar^{\mathrm{ev}_0 }[rd]  \ar_{\mathrm{ev}_+}[ld] & \\
L^{k} && L   }
\end{equation} 
we define
$$
\frak{m}_{k, \beta} \colon \Omega^\bullet (L)^{\otimes k}  \to  \Omega^\bullet (L), \quad 
\frak{m}_{k,\beta} (x_1, \cdots, x_k) = (\mathrm{ev}_0)_! \, \mathrm{ev}_{+}^* (\pi_1^* x_1 \wedge \cdots \wedge \pi_k^* x_k).
$$
We then define 
$$
\frak{m}_k \coloneqq \sum_{\beta} \frak{m}_{k, \beta} \, T^{\omega(\beta) / 2 \pi}.
$$ 
The sequence $\{\frak{m}_k\}_{k \geq 0}$ defines an $A_\infty$-structure on the completed de Rham complex $\Omega^\bullet(L; \Lambda_+)$. By passing to harmonic forms, it induces an $A_\infty$-structure on the cohomology $H^\bullet(L; \Lambda_+)$. This resulting $A_\infty$-algebra$$
(H^\bullet(L; \Lambda_+), \{\frak{m}^\mathrm{can}_k\}_{k \geq 0})
$$
is called a \emph{canonical} (a.k.a. \emph{minimal}) \emph{model}, see \cite[Theorem 8.3]{Fuk10} for the construction. For simplicity, we omit the superscript $\mathrm{can}$ throughout the remainder of this section.

Let $b \in H^1(L;\Lambda_0)$, and decompose it as $b = b_0 + b_+$, where $b_0 \in H^1(L;\C)$ and $b_+ \in H^1(L;\Lambda_+)$. The deformed operations are defined by
\begin{equation}\label{equ_deformb+}
\frak{m}_{k, \beta}^{b_+} (x_1, \cdots, x_k) \coloneqq \frak{m}_{\bullet, \beta} (e^{b_+}, x_1, e^{b_+}, \cdots, e^{b_+}, x_k, e^{b_+} ).
\end{equation}

We further incorporate the deformation of a (non-unital) flat line bundle $\frak{L}$ over $L$. Choose a basis $(\theta_i)_{i=1}^m$ for $H_1(L; \Z)$,and let $(e_i)_{i=1}^m$ be the dual basis of $(\theta_i)_{i=1}^m$ for $H^1(L; \Z) \simeq \mathrm{Hom}(H_1(L; \Z), \Z)$. Writing $b_0 = \sum b_{0,i} e_i$, the associated holonomy representation of $\frak{L}$ is a homomorphism
$$
\rho_{b_0} \colon H_1(L; \Z) \to \C \backslash \{0\} \quad \mbox{ defined by $\theta \mapsto \exp (\theta \cap b_0)$}.
$$
We then define
\begin{equation}\label{equ_deformb0}
\frak{m}_{k, \beta}^b \coloneqq \rho_{b_0} (\partial \beta) \cdot \frak{m}_{k,\beta}^{b_+} \, \mbox{ and } \,
\frak{m}_k^b \coloneqq \sum_{\beta} \frak{m}_{k, \beta}^b.
\end{equation}
The sequence $\{ \frak{m}^b_k\}_{k \geq 0}$ defines an $A_\infty$-algebra structure on the de Rham cohomology $H^\bullet (L; \Lambda_0)$, see \cite{FOOOToric1, Fuk10}.

To obtain a weakly unobstructed $A_\infty$-algebra, we impose the following positivity assumption.

\begin{definition}
Let $(X, \omega)$ be a symplectic manifold, $L$ a Lagrangian torus, and $J$ a compatible almost complex structure.
\begin{enumerate}
\item 
The pair $(L,J)$ is called \emph{positive} if every nonzero class $\beta$ with $\mathcal{M}(X,L,J,\beta) \neq \emptyset$ has Maslov index at least two. 
\item The pair $(X,J)$ is called \emph{symplectically Fano} if every nonconstant $J$-holomorphic sphere class $\alpha$ has positive first Chern number. 
\end{enumerate}
When the choice of $J$ is clear from the context, we simply say that $L$ is positive and $X$ is symplectically Fano.
\end{definition}

\begin{remark}
For example, every monotone symplectic manifold is symplectically Fano and every monotone Lagrangian submanifold is positive.
\end{remark}

Assume that $L$ is positive, $J$ is regular, and $\mathrm{ev}_0 \colon \mathcal{M}_1(X, L, J, \beta) \to L$ is submersive. For each homotopy class $\beta$ of Maslov index two, $\mathcal{M}_1(X, L, J, \beta) = \overline{\mathcal{M}}_1(X, L, J, \beta)$.
Moreover,
\begin{equation}\label{equ_mobeta1}
\frak{m}^\mathrm{can}_{0, \beta}(1) = \frak{m}_{0, \beta}(1) = (\mathrm{ev}_{0})_!(\mathrm{PD} [\mathcal{M}_1(X, L, J, \beta)]) = n_\beta \cdot T^{\omega(\beta) / 2 \pi} \, \mathrm{PD}[L] \in H^0 (L; \Lambda_+),
\end{equation}
where $n_\beta$ is defined as the degree of the map~\eqref{equ_degreeev}. Since $\mathrm{ev}_0$ is submersive, the degree can be computed directly via integration along fiber without additional perturbation data. 

For each integer $k > 0$, let 
$$
\frak{forget}_{0} \colon \overline{\mcal{M}}_{k+1}(X, L, J, \beta) \to \overline{\mcal{M}}_{1}(X, L, J, \beta) 
$$
be the map forgetting the boundary marked points $z_1, z_2, \cdots, z_k$. We choose perturbations compatible with the forgetful maps, see \cite{FOOO24} for the construction. Then, for $b \in H^1_{dR}(L; \Lambda_0)$, the forgetful-map compatibility yields
\begin{equation}\label{equ_mkbetab6k}
\frak{m}_{k,\beta} (b^{\otimes k}) =  \rho_{b_0} (\partial \beta) \cdot \frac{ (\partial \beta \cap b_+)^k }{k!}  \frak{m}_{0, \beta}(1).
\end{equation}

Since $\beta$ is regular and $\mathrm{ev}_0$ is submersive, no perturbation is required for $\mcal{M}_{1}(X, L, J, \beta)$ so that $\frak{m}_{0, \beta}(1)$ is given by~\eqref{equ_mobeta1}. Combining~\eqref{equ_deformb+},~\eqref{equ_deformb0}, ~\eqref{equ_mobeta1} and~\eqref{equ_mkbetab6k}, we obtain
\begin{align*}
\frak{m}_0^b(1) &= \sum_{k=0}^\infty \frak{m}_{k,\beta} (b^{\otimes k}) = \rho_{b_0} (\partial \beta)  \cdot \sum_{k=0}^\infty  \frak{m}_{k,\beta} (b_+^{\otimes k}) \\
&=  \rho_{b_0} (\partial \beta) \cdot \sum_{k=0}^\infty \frac{ (\partial \beta \cap b_+)^k }{k!}  \frak{m}_{0, \beta}(1) \\
&= \exp( \partial \beta \cap b_0) \cdot \exp \left( \partial \beta \cap b_+ \right) \cdot n_\beta \cdot T^{\omega(\beta) / 2 \pi} \mathrm{PD}[L].
\end{align*}

\begin{definition}
The \emph{disk potential} of a Lagrangian torus $L$ is the function  
$$
W \colon H^1(L; \Lambda_0) \to \Lambda \quad \mbox{ defined by $\frak{m}_0^b(1) = W(b) \cdot \mathrm{PD}[L]$.}
$$
\end{definition}

Using the chosen basis $(e_i)_{i=1}^m$ for $H^1(L; \Z)$ and setting $y_i \coloneqq \exp (e_i)$, the potential function $W(b)$ can be expressed as a Laurent polynomial 
\begin{equation}\label{equ_diskpotform}
W(\mathbf{y})= \sum_{\beta \in \pi_2(X,L)} n_\beta \cdot \prod_{i=1}^m y_i^{\partial \beta \, \cap \, e_i} \cdot T^{\omega(\beta) / 2 \pi}.
\end{equation}
This Laurent polynomial is also referred to as the \emph{disk potential} of $L$.

\subsection{Stability for moduli spaces of holomorphic disks}\label{sec_stabilityformoduli}

We now return to the setting of Section~\ref{sec_sympredGIT}. In Corollary~\ref{cor_countinginv}, we compared the moduli spaces of holomorphic disks in $X^{ss}$ and in the GIT quotient $X \git \mathbb{G}$. In this subsection, we study the relationship between moduli spaces of holomorphic disks in $X$ and in $X^{ss}$ in~\eqref{equation_twostep}. 

When taking the GIT quotient, unstable points are discarded and only semistable points descend to the quotient. Consequently, not every holomorphic disk in $X$ gives rise to a holomorphic disk in $X \git \mathbb{G}$. This motivates the following notion of stability for moduli spaces.

\begin{definition}
Let $\beta \in \pi_2(X, L)$ be a relative homotopy class.
\begin{itemize}
\item The moduli space $\mcal{M}(X,L,J,\beta)$ is called \emph{semistable} if every holomorphic disk $\varphi \in \mcal{M}(X,L,J,\beta)$ has image contained entirely in the semistable locus, that is,
$$
\varphi(\mathbb{D}) \subseteq X^{ss} \mbox{ for all $\varphi \in \mcal{M}(X, L, J, \beta)$}.
$$
\item The moduli space $\mcal{M}(X,L,J,\beta)$ is called \emph{unstable} if it is \emph{not} semistable.
\end{itemize}
When $(X,L,J)$ is clear from the context, we refer to $\beta$ as semistable or unstable, according to the stability of its moduli space. 
\end{definition}

Let $\iota \colon X^{ss} \to X$ be the inclusion. It induces a homomorphism
$$
\iota_* \colon \pi_2(X^{ss}, L) \to \pi_2(X, L), \quad \alpha \mapsto \iota_* \alpha.
$$
In general, this map is neither injective nor surjective. After compactification, additional homotopy classes whose representatives intersect the unstable locus $X^{us}$ may appear. Moreover, distinct classes in $\pi_2(X^{ss},L)$ may become homotopic in $X$. Nevertheless, if two homotopy classes coincide after compactification, their boundary classes must agree.

\begin{lemma}\label{lemma_boundarysame}
Let $\alpha_1, \alpha_2 \in \pi_2(X^{ss}, L)$. 
If $\iota_* \alpha_1 =  \iota_* \alpha_2$, then $\partial \alpha_1 = \partial \alpha_2$ in $\pi_1(L)$.
\end{lemma}

\begin{proof}
Since $\iota_*$ induces an isomorphism on $\pi_1(L)$, we have
$$
\iota_* (\partial \alpha_1) = \partial (\iota_* \alpha_1) = \partial (\iota_* \alpha_2) = \iota_* (\partial \alpha_2).
$$  
\end{proof}

Since our primary interest lies in disk counting invariants, we introduce the sets
\begin{itemize}
\item
$\Pi(X,L,J) \coloneqq \{\beta \in \pi_2(X,L) \mid \mathcal{M}(X,L,J,\beta) \neq \emptyset  \}$,
\item
$\Pi^{(2)}(X,L,J) \coloneqq \{\beta \in \Pi(X,L,J) \mid \mu(\beta) = 2 \}$.
\end{itemize} 
When the choice of $J$ is fixed, we simply write $\Pi(X,L) \coloneqq \Pi(X,L,J)$ and $\Pi^{(2)}(X, L) \coloneqq \Pi^{(2)}(X,L,J)$. Let $\beta \in \pi_2(X, L)$ be a semistable homotopy class. For every holomorphic disk $\varphi \in \mathcal{M}(X,L,J,\beta)$, there exists a holomorphic disk $\varphi^\prime$ in $X^{ss}$ such that $\varphi = \iota \circ \varphi^\prime$. We define 
$$
I^{ss}_\beta \coloneqq \{ \alpha \in \Pi(X^{ss}, L) \mid \iota_* \alpha = \beta\}.
$$ 
For any $\varphi \in \mcal{M}(X^{ss},L,J,\alpha)$ with $\iota_*\alpha=\beta$, the bundle pairs $((\iota \circ \varphi)^* TX,(\iota \circ \varphi)^* TL)$ and $(\varphi^* TX^{ss}, \varphi^* TL)$ are canonically identified. It follows that
\begin{enumerate}
\item the Maslov index of $\alpha$ is equal to that of $\iota_* \alpha$, and
\item if $\beta$ is regular, then $\alpha$ is also regular.
\end{enumerate} 
Assume furthermore that $L$ is spin and equip $L$ with the same spin structure in both $X$ and $X^{ss}$. We then obtain the following proposition.

\begin{proposition}\label{proposition_coprodmapsto}
Let $\beta \in \Pi^{(2)}(X,L)$ be regular and semistable. Then the map 
$$
\iota_* \colon \coprod_{ \alpha \in I^{ss}_\beta} \mcal{M}(X^{ss}, L, J, \alpha) \to \mcal{M}(X, L, J, \beta), \quad [\varphi] \mapsto [\iota \circ \varphi].
$$
is an orientation-preserving diffeomorphism.
\end{proposition}

We obtain the following corollaries.

\begin{corollary}\label{cor_nbetanalpha}
Let $\beta \in \Pi^{(2)}(X,L)$ be regular and semistable. Assume that $\mcal{M}(X, L, J, \beta) = \overline{\mcal{M}}(X, L, J, \beta)$. Then
$$
n_\beta = \sum_{\alpha \in I^{ss}_\beta} n_{\alpha}.
$$
\end{corollary}

Suppose moreover that the induced map $\iota_* \colon \Pi^{(2)} (X^{ss}, L) \to \Pi^{(2)} (X, L)$ is injective. Then, for each $\beta \in \Pi^{(2)}(X,L)$, there exists a unique class $\alpha \in \Pi^{(2)}(X^{ss}, L)$ such that $\iota_* (\alpha) = \beta$. Equivalently, $\#(I^{ss}_\beta) = 1$.

\begin{corollary}
For $\alpha \in \Pi^{(2)} (X^{ss}, L)$, suppose that $\iota_* \alpha$ is regular and semistable. Assume that $\mcal{M}(X, L, J, \beta) = \overline{\mcal{M}}(X, L, J, \beta)$ and
\begin{equation}\label{equ_inclusionofeffdiskclasses}
\iota_* \colon \Pi^{(2)} (X^{ss}, L) \to \Pi^{(2)} (X, L) 
\end{equation}
is injective. Then 
$$
n_\alpha = n_{\iota_*(\alpha)}.
$$
\end{corollary}

Because of Proposition~\ref{proposition_coprodmapsto}, for a semistable class $\beta \in \pi_2(X,L)$, we introduce the notation
\begin{equation}\label{equ_mXssnot}
\mcal{M}(X^{ss}, L, J, \beta) \coloneqq \coprod_{ \alpha \in I^{ss}_\beta} \mcal{M}(X^{ss}, L, J, \alpha).
\end{equation}

\subsection{Semistable disk potential and GIT quotients}

Our goal is to compare the disk potential of $L$ in $X$ with that of $K = L / G$ in the GIT quotient $M \coloneqq X \git \mathbb{G}$. If a homotopy class $\beta \in \pi_2(X, L)$ is unstable, then there exists a holomorphic disk $\varphi \in \mathcal{M}(X, L, J, \beta)$ whose image intersects the unstable locus $X^{us}$. Such a disk disappears upon restriction to the semistable locus $X^{ss}$ and therefore does not descend to a holomorphic disk in $M$.  These disks may cause compactness issues$\colon$ even if $\mathcal{M}(X, L, J, \beta)$ is compact, the corresponding moduli space in $X^{ss}$ need not be compact.

To exclude this bad situation, we introduce the following notion.

\begin{definition}
A relative homotopy class $\beta \in \pi_2(X,L)$ is called \emph{fully excluded} if every holomorphic disk in $\mcal{M}(X, L, J, \beta)$ intersects the unstable locus, that is,
$$
\varphi(\mathbb{D}) \cap X^{us} \neq \emptyset \mbox{ for all $\varphi \in \mcal{M}(X, L, J, \beta)$}.
$$
\end{definition}

From now on, we assume that $L$ is torus, $G$ is a torus, and the quotient $K \coloneqq L/G$ is also a torus. Let $(\theta_i)_{i=1}^m$ be a basis of $H_1(L;\Z) \simeq \pi_1(L)$ and consider the induced homomorphism 
$$
q_* \colon \pi_1(L) \to \pi_1(K).
$$
Since the $G$-action is free, the long exact sequence of homotopy groups associated with the principal $G$-bundle implies that the kernel of $q_*$ is a primitive sublattice. Since 
$$
\pi_1(L) \simeq \pi_1(K) \times \pi_1(G)
$$
After replacing the basis $(\theta_i)_{i=1}^m$ if necessary, we may assume that 
\begin{enumerate}
\item $(\theta_i)_{i={s+1}}^m$ is a basis for $\mathrm{ker}(q_*)$ and
\item $(\theta_i)_{i=1}^s$ projects to a basis for $\pi_1(K)$.
\end{enumerate}

Let $(f_i)_{i=1}^s$ be the dual basis for $H^1(K; \Z)$. Extending this basis, we obtain a dual basis $(e_i)_{i=1}^m$ of $H^1(L;\Z)$ satisfying $q^* f_i =  e_i$ for $i = 1, \dots, s$. Let $\mathbf{y} = (y_i)$ (resp. $\mathbf{z} = (z_i)$) be the variables corresponding to $\exp(e_i)$ for $L$ (resp. $\exp(f_i)$ for $K$). Then the disk potential of $L$ is expressed as a Laurent polynomial in the variables $\mathbf{y}$.

We now define the \emph{semistable disk potential} $W_L^{ss}$ of $L$. Recall that $n_\beta$ can be nonzero only if $\mu_L (\beta) = 2$. We classify all semistable disk classes $\beta$ with $\mu_L (\beta) = 2$. We restrict attention to semistable disk classes of Maslov index two.

\begin{definition}
The \emph{semistable disk potential} $W_L^{ss}$ of $L$ is defined by
$$
W_L^{ss}(\mathbf{y}) \coloneqq \sum_{\substack{\beta \in \pi_2(X,L);\\ \beta \mbox{\scriptsize{ is semistable}}}} n_\beta \cdot \prod_{i=1}^m y_i^{\partial \beta \, \cap \,  e_i} \cdot T^{\omega(\beta) / 2 \pi}.
$$
\end{definition}

Our choice of basis determines the relations
\begin{equation}\label{equ_choiceofexp}
\begin{cases}
z_i = y_i &\mbox{ for $i = 1, \cdots, s$} \\
y_i = 1 &\mbox{ for $i = s+1, \cdots, m$.}
\end{cases}
\end{equation}
Substituting these relations into $W_L^{ss}$ yields a Laurent polynomial in the variables $\mathbf{z}$, which we denote by
$$
W_L^{ss} |_{\mathrm{ker}(q_*)} (\mathbf{z}).
$$

We state the main theorem of this section, which computes the disk potential of $K$.

\begin{theorem}[Theorem~\ref{thmx_diskGIT}]\label{theorem_diskGIT}
Let $G$ be a torus and let $\mathbb{G}$ denote its complexification. Suppose that $G$ acts linearly on a complex vector space $V$. Let $X \subseteq \mathbb{P}(V)$ be a $\mathbb{G}$-invariant projective variety. Let $\mu \colon X \to \frak{g}^*$ be a moment map of the Hamiltonian $G$-action on $X$. Let $L$ be an  oriented $G$-invariant Lagrangian torus of $X$ contained in $\mu^{-1}(0)$. Assume that
\begin{enumerate}
\item $G$ acts freely on the level set $\mu^{-1}(0)$ (and hence on $L$),
\item $L$ is positive and is endowed with the product spin structure induced by the spin structure on $L/G$,
\item $J$ is regular and every unstable homotopy class $\beta \in \Pi^{(2)}(X, L)$ is fully excluded, 
\item $\mathrm{ev}_0 \colon \mathcal{M}_1(X,L, J, \beta) \to L$ is submersive for every $\beta \in \Pi^{(2)}(X^{ss}, L)$,
\item both $X$ and $X \git G \simeq X \git \mathbb{G}$ are symplectically Fano.
\end{enumerate}
Then the disk potential of ${L/G}$ is equal to $W_L^{ss} |_{\mathrm{ker}(q_*)} (\mathbf{z})$, that is,
$$
W_{L / G} (\mathbf{z}) = W_L^{ss} |_{\mathrm{ker}(q_*)} (\mathbf{z}).
$$
\end{theorem}

Before proving Theorem~\ref{theorem_diskGIT}, we collect several auxiliary lemmas.

\begin{lemma}\label{lemma_positivity}
The Lagrangian $L$ is positive in $X^{ss}$ if and only if $K = L/G$ is also positive in $M$. 
\end{lemma}

\begin{proof}
Suppose that there exists a nonconstant holomorphic disk in $M$ with boundary on $K$ and Maslov index $\leq 0$. By Proposition~\ref{proposition_existenceoflift}, such a disk admits a holomorphic lift to $X^{ss}$ and by Lemma~\ref{lemma_qregular}, the lifted disk has the same Maslov index. This contradicts the positivity of $L$ in $X^{ss}$. Hence $K$ is positive. The converse immediately follows from Lemma~\ref{lemma_qregular}.
\end{proof}

\begin{lemma}\label{lemma_submerstive}
The evaluation map $\mathrm{ev}_0 \colon \mathcal{M}_1(P_\mathbb{G}, L,J,\beta) \to L$ is submersive if and only if $\mathrm{ev}_0 \colon \mathcal{M}_1(M,K,J,\beta) \to K$ is submersive.
\end{lemma}

\begin{proof}
By construction, $T_p L \simeq T_{q(p)}K \oplus \frak{g}$. The claim therefore follows from the commutative diagrams~\eqref{equ_commutativitiyofcompletelyinte2} and~\eqref{equ_commutativitiyofcompletelyinte3}.
\end{proof}

\begin{proof}[Proof of Theorem~\ref{theorem_diskGIT}]
Let $\beta \in \Pi^{(2)}(X^{ss},L)$. Since $L$ is positive and $J$ is regular, we have
$$
\mcal{M}_1(X,L,J, \iota_* \beta) = \overline{\mcal{M}}_1(X,L,J,\iota_* \beta)
$$ 
and this moduli space is a compact manifold without boundary of dimension $L$. By Proposition~\ref{proposition_coprodmapsto} and~\eqref{equ_mXssnot}, the inclusion map $\iota \colon X^{ss} \hookrightarrow X$ induces an orientation preserving diffeomorphism between $\mcal{M}_1(X^{ss},L,J, \iota_* \beta)$ and $\mcal{M}_1(X,L,J, \iota_*\beta)$.

The Lagrangian torus $L$ is positive in $X^{ss}$. If there existed a nonconstant holomorphic disk in $X^{ss}$ with Maslov index $\leq 0$, then its image under $\iota$ would define a holomorphic disk in $X$ with the same Maslov index, contradicting the positivity of $L$ in $X$. By Lemma~\ref{lemma_positivity}, the quotient torus $K = L/G$ is also positive. Therefore the disk potential of $K$ is well-defined. By Lemma~\ref{lemma_isoonhomotopy} and~\eqref{equ_diskpotform}, it takes the form
\begin{equation}\label{equ_WL/G}
W_{L / G} (\mathbf{z}) = \sum_{\beta \in \pi_2(X^{ss}, L)} n_{q_* \beta} \cdot \prod_{i=1}^s z_i^{\partial (q_*\beta) \, \cap \, f_i} \cdot T^{\omega(q_* \beta) / 2 \pi}.
\end{equation}

Since $M$ is symplectically Fano, sphere bubbling does not contribute to $W_{L / G}$. Hence Theorem~\ref{theorem_fundacycle}, Corollary~\ref{cor_countinginv}, and Lemma~\ref{lemma_submerstive} imply that
\begin{equation}\label{equ_openGWcorres}
n_\beta = n_{q_* \beta}.
\end{equation}
Moreover, by Lemma~\ref{lemma_qregular} and the definition of the reduced symplectic form,
\begin{equation}\label{equ_omegamu}
\mu(q_* \beta) = \mu(\beta) \mbox{ and } \omega(q_* \beta) = \omega(\beta).
\end{equation}

Since $q^* \colon \Z \langle f_1, \cdots, f_s \rangle \to \Z \langle e_1, \cdots, e_s \rangle$ is isomorphic, we have
$$
\partial (q_* \beta) \cap f_i = q_*(\partial \beta \cap q^* f_i)  = q_*(\partial \beta \cap e_i) \,\, \mbox{ for $i =1, \cdots, s$.}
$$
Hence
\begin{equation}\label{equ_monomials}
\prod_{i=1}^s z_i^{\partial (q_* \beta) \, \cap \, f_i } = \prod_{i=1}^s z_i^{q_*(\partial \beta \, \cap \, e_i)} = \prod_{i=1}^m y_i^{\partial \beta \, \cap \, e_i } |_{_{\mathrm{ker}(q_*)}}.
\end{equation}

Let $\beta \in \pi_2(X^{ss}, L)$. Since every unstable effective homotopy class is fully excluded, the class $\iota_* \beta$ is semistable.  By Lemma~\ref{lemma_boundarysame}, the boundary classes agree whenever two relative classes become identified in $X$. Furthermore, Corollary~\ref{cor_nbetanalpha} shows that the corresponding disk counting invariants also agree. Consequently, neither the monomials nor the coefficients change after passing from $X^{ss}$ to $X$. Combining~\eqref{equ_openGWcorres},~\eqref{equ_omegamu}, and~\eqref{equ_monomials}, we obtain
$$
W_{L / G} (\mathbf{z}) = \sum_{\beta} n_{\beta} \cdot \prod_{i=1}^m y_i^{\partial \beta \, \cap \, e_i} |_{\mathrm{ker}(q_*)}  \cdot T^{\omega(\beta) / 2 \pi} = W_L^{ss} |_{\mathrm{ker}(q_*)} (\mathbf{z}),
$$
where the summation is taken over all semistable disk classes $\beta \in \pi_2(X,L)$ of Maslov index two.
\end{proof}

\begin{remark}
If $X$ is not symplectically Fano, then Theorem~\ref{theorem_diskGIT} only determines the leading-order contribution to the disk potential. In general, a holomorphic disk with sphere bubbles may contributes to the disk potential.

Moreover, the semistable locus depends on the choice of moment-map level set, and hence so does the corresponding symplectic reduction. As a consequence, the semistable disk potential and the disk potential of the quotient may vary with the choice of reduction.
\end{remark}

\section{Examples: going-up and going-down}

In this section, we illustrate how Corollary~\ref{cor_countinginv} and Theorem~\ref{theorem_diskGIT} can be applied in concrete examples to compute disk potentials.

\begin{example}
Consider the complex projective plane $X = \CP^2$ with homogeneous coordinates $[x_0 : x_1 : x_2 ]$, equipped with the $T^2$-action
\begin{equation}\label{equ_T2actoncp2}
\bigl(e^{i \theta_1}, e^{i \theta_2}\bigr) \, * [x_0 : x_1 : x_2 ] = \bigl[x_0 : e^{i \theta_1} x_1 : e^{i \theta_2} x_2 \bigr].
\end{equation}
We normalize the Fubini–Study K\"{a}hler form so that the moment polytope of the associated moment map $\mu_{T^2} \colon X \to \R^2$ is the convex hull $\Delta$ of the points $(-1,-1), (-1,2)$, and $(2,-1)$. Let $L(\mathbf{u})$ denote the Lagrangian toric fiber over $\mathbf{u} = (u_1, u_2) \in \mathrm{Int}(\Delta)$. By \cite{Cho04}, the set $\Pi^{(2)}(X,L(\mathbf{u}))$ consists of three homotopy classes $\{\beta_0, \beta_1, \beta_2\}$ where $\beta_i$ intersects the toric divisor $V(x_i)$ exactly once. The disk potential of $L(\mathbf{u})$ is 
$$
W_{L(\mathbf{u})} (\mathbf{y}) = y_1 T^{1+u_1} + y_2 T^{ 1+ u_2}  + \frac{1}{y_1y_2} T^{1-u_1-u_2},
$$
where $y_i$ is the exponential variable corresponding to an $S^1_{\theta_i}$-orbit for $i = 1,2$.

Next, consider the $S^1$-subaction defined by
$$
e^{i \theta_1} * [x_0 : x_1 : x_2 ] = [x_0 : e^{i \theta_1} x_1 : x_2 ].
$$ 
Its moment map is the first component of $\mu_{T^2}$, that is,
$$
\mu_{S^1} \coloneqq \mathrm{pr}_1 \circ \mu_{T^2}, \quad \mu_{S^1}([x_0 : x_1 : x_2]) = \frac{3|x_1|^2}{|x_0|^2+|x_1|^2+|x_2|^2} - 1.
$$ 
We perform symplectic reduction at the level set $\mu_{S^1}^{-1}(0)$. The induced $S^1$-action on $\mu_{S^1}^{-1}(0)$ is free. The semistable locus is the complement of the divisor $V(x_1)$ and the fixed point $[0 : 1 : 0]$. 

Let $\mathbf{u}\in \mathrm{Int}(\Delta)$ satisfy $u_1=0$. Then the semistable disk potential is 
$$
W_{L(\mathbf{u})}^{ss}(\mathbf{y}) = y_2 \, T^{1+ u_2}  + \frac{1}{y_1y_2} \, T^{1-0-u_2}.
$$
Since the reduction removes the direction corresponding to $y_1$, we set $z \coloneqq y_2$ and $y_1 \coloneqq 1$. By Theorem~\ref{theorem_diskGIT},  the disk potential of the quotient Lagrangian circle $L(\mathbf{u})/S^1$ is
\begin{equation}\label{equ_WLuS1}
W_{L(\mathbf{u})/S^1}(\mathbf{z}) = z \, T^{1+ u_2} + \frac{1}{z} \, T^{1-u_2}.
\end{equation}

Observe that the $S^1$-action on $\mu_{S^1}^{-1}(0) \simeq S^3$ is the standard diagonal Hopf action. Therefore, $X \, \git \, S^1 \simeq \CP^1$ and $L(\mathbf{u}) \, / \, S^1 \simeq S^1$. The disk potential in~\eqref{equ_WLuS1} coincides with the disk potential of a toric circle fiber in $\CP^1$.
\end{example}

We next discuss a non-toric example.

\begin{example}
Let $D \coloneqq \mathrm{diag}(5,5,0,0,0)$ be a $(5 \times 5)$ diagonal matrix. The (co)adjoint orbit $X$ of $D$ under the conjugate action of $\mathrm{U}(5)$ is isomorphic to $\mathrm{Gr}(2, \C^5)$. It carries a completely integrable system, called the \emph{Gelfand--Zeitlin system},  
$$
\Phi = (\Phi_{i,j} \mid 1 \leq i \leq 3, 1 \leq j \leq 2) \colon X \to \R^{3 \times 2}
$$
where $\Phi_{i,j}(A)$ is defined as the $j$-th largest eigenvalue of the $(i+j-1)\times(i+j-1)$ principal submatrix of $A \in X$, see \cite{Thi81, GS83}. The image of $\Phi$ is the Gelfand--Zeitlin polytope $\Delta$ determined by the interlacing inequalities arising from the min–max principle. Explicitly, $\Delta$ is given as follows.
\begin{center}
\begin{tikzpicture}[scale=1, every node/.style={font=\small}]

\node (l1) at (2,1.95) {$u_{3,2}$};
\node (l1) at (4,1.95) {$u_{2,2}$};
\node (l1) at (6,1.95) {$u_{1,2}$};
\node (l1) at (8,2) {$0$};

\node (l1) at (0,1) {$5$};
\node (l1) at (2,0.95) {$u_{3,1}$};
\node (l1) at (4,0.95) {$u_{2,1}$};
\node (l1) at (6,0.95) {$u_{1,1}$};

\node at (1,1) {$\geq$};
\node at (3,1) {$\geq$};
\node at (5,1) {$\geq$};

\node at (3,2) {$\geq$};
\node at (5,2) {$\geq$};
\node at (7,2) {$\geq$};

\node[rotate=90]  at (2,1.5) {$\geq$};
\node[rotate=90] at (4,1.5) {$\geq$};
\node[rotate=90] at (6,1.5) {$\geq$};

\end{tikzpicture}
\end{center}

Let $\mathbf{u} \in \mathrm{Int}(\Delta)$ be a point and denote by $L(\textbf{u})$ the fiber of $\Phi$ over $\mathbf{u}$. Then $L(\mathbf{u})$ is a Lagrangian torus. By \cite{NNU08}, the disk potential of $L(\mathbf{u})$ is
\begin{align*}
W_{L(\textbf{u})}(\mathbf{y}) &= \frac{1}{y_{3,1}}T^{5-u_{3,1}} + \frac{y_{3,1}}{y_{2,1}}T^{u_{3,1}-u_{2,1}}  + \frac{y_{2,1}}{y_{1,1}} T^{u_{2,1}-u_{1,1}} + \frac{y_{3,2}}{y_{2,2}}  T^{u_{3,2}-u_{2,2}}\\
&+ \frac{y_{2,2}}{y_{1,2}} T^{u_{2,2}-u_{1,2}} + y_{1,2} T^{u_{1,2}} + \frac{y_{3,1}}{y_{3,2}} T^{u_{3,1}-u_{3,2}} + \frac{y_{2,1}}{y_{2,2}} T^{u_{2,1}-u_{2,2}} + \frac{y_{1,1}}{y_{1,2}} T^{u_{1,1}-u_{1,2}}.
\end{align*}

Let $T^\prime$ be the diagonal maximal torus of $\mathrm{U}(5)$. Since the diagonal $S^1$ acts trivially on $X$, we define 
$$
T \coloneqq T^\prime / \mathrm{diag}(S^1) \simeq T^4
$$
to obtain an effective action. We identify $T$ with the subgroup $\{ \mathrm{diag}(e^{i \theta_1}, \cdots, e^{i \theta_4}, 1) \}$. The moment map $\mu_T$ for the $T$-action is given by 
\begin{equation}\label{equ_momentmapT}
\mu_T(A) = 
(u_{1,1}, u_{2,1} + u_{1,2} - u_{1,1}, u_{3,1} + u_{2,2} - u_{2,1} - u_{1,2}, 5+u_{3,2} - u_{3,1} - u_{2,2})
\end{equation}
for $A \in L(\mathbf{u})$. Let $\nu_0 \coloneqq (2,2,2,2)$ and consider the level set $\mu_T^{-1}(\nu_0)$. This level set consists of Gelfand--Zeitlin fibers. Moreover, the stabilizers of any two points lying in the relative interior of the same face of $\Delta$ are isomorphic. Therefore, to determine the stabilizers on $\mu_T^{-1}(\nu_0)$, it suffices to examine the stabilizers of only finitely many fibers.

The intersection of the affine hyperplanes determined by $\nu_0$ with $\Delta$ meets the relative interiors of eight faces of $\Delta$. Among them, there are two $3$-dimensional faces $f_1$ and $f_2$. Using the description of Gelfand--Zeitlin fibers in \cite{CKO20}, one sees that the fiber over each point in the relative interior of $f_i$ contains an $S^3$-factor in addition to a torus factor $T^3$, even though $\dim f_i = 3$. This spherical factor eliminates the residual isotropy and ensures that the $T$-action is free on $\mu_T^{-1}(\nu_0)$.

Moreover, the level set $\mu_T^{-1}(\nu_0)$ intersects every facet of $\Delta$ and the inverse image of the relative interior of each facet is contained in the semistable locus $X^{ss}$. Consequently, every holomorphic disk of Maslov index two lies in $X^{ss}$. Hence every class in $\Pi^{(2)}(X,L(\mathbf{u}))$ is semistable and 
$$
W^{ss}_{L(\mathbf{u})}(\mathbf{y}) = W_{L(\mathbf{u})}(\mathbf{y}).
$$

Let $\mathbf{u} \in \mathrm{Int}(\Delta)$ satisfy $\mu_T(\mathbf{u}) = \nu_0$. Set $z_1 = y_{1,2}, z_2 = y_{2,2}, v_1 = u_{1,2}$, and $v_2 =u_{2,2}$. Using the relations from~\eqref{equ_momentmapT}, we have
$$
y_{1,1} = y_{2,1} y_{1,2} = y_{3,1} y_{2,2} = y_{3,2} = 1.
$$
Therefore, Theorem~\ref{theorem_diskGIT} yields the disk potential of the quotient Lagrangian $L(\mathbf{u})/T$
\begin{align*}
W_{L(\mathbf{u}) / T}(\mathbf{z}) &= {z_2} T^{v_2-1} + \frac{z_1}{z_2} T^{2+v_1-v_2}+ \frac{2}{z_1}T^{2-v_1} \\ &+ \frac{2}{z_2}T^{3-v_2} + \frac{z_2}{z_1} T^{v_2 - v_1} + z_1 T^{v_1} + \frac{1}{z_1z_2}T^{4-v_1-v_2}. 
\end{align*} 
\end{example}

\begin{remark}
This generalizes a result in \cite{KLZ24}, where the disk potential was computed only for the \emph{monotone} Gelfand--Zeitlin fiber. In contrast, the above computation determines the disk potential for an arbitrary torus fiber. In forthcoming joint work with J.~Choi, we systematically study torus actions on Gelfand--Zeitlin systems and extend these constructions to other Grassmannians of $2$-planes.
\end{remark}

\begin{remark}
A level $\lambda \in \mathfrak{t}^*$ is said to \emph{interact with all effective disk classes} if the complex orbit of the level set contains all holomorphic disks of Maslov index two. Equivalently, every $\beta \in \Pi^{(2)}(X,L)$ is semistable, so that $W = W^{ss}$.

In general, in the compact toric setting, it is hard to simultaneously satisfy both the freeness of the $G$-action on the level set and the condition that the level intersects all classes in $\Pi^{(2)}(X,L)$. For Grassmannians, however, vanishing cycles may appear along lower-dimensional strata, making it possible for both conditions to hold simultaneously. The above example illustrates such a phenomenon.
\end{remark}

We apply Corollary~\ref{cor_countinginv} to recover the disk potential of the total space from that of the quotient. The following elementary example illustrates this “going-up” procedure.

\begin{example} 
Let $n \in \mathbb{N}_{\geq 0}$ and consider the line bundle $\mathcal{O}(-n)$ over the projective line $\CP^1$. Its total space is a noncompact toric manifold whose moment polytope is
$$
\left\{ \textbf{u} = (u_1, u_2) \in \R^2 \mid u_1 \geq 0, u_2 \geq 0, nu_2 - u_1 + 1 \geq 0 \right\}.
$$ 
Let $\mu_{T^2}$ denote the moment map. The second component $\mu_{S^1} \coloneqq \mathrm{pr}_2 \circ \mu_{T^2}$ generates the fiberwise $S^1$-action on $\mathcal{O}(-n)$. Moreover, the complement of the zero section $o_{\CP^1}$ is a holomorphic principal $\mathbb{C}^*$-bundle over $\CP^1$
$$
q \colon P = \mcal{O}(-n) \backslash {o}_{\CP^1} \to \CP^1.
$$ 

Fix $\nu \in \mathbb{N}$. The $S^1$-action is free on the level set $\mu_{S^1}^{-1}(\nu)$, which is diffeomorphic to a lens space. The corresponding symplectic reduction is isomorphic to $\CP^1$ equipped with a multiple of the Fubini–Study form. Let $\mathbf{u}=(u_1,u_2)$ with $u_2=\nu$. We denote by $L(\mathbf{u})$ the toric Lagrangian fiber over $\mathbf{u}$. Then the disk potential of the quotient Lagrangian $L(\mathbf{u})/S^1$ is
$$
W_{L(\mathbf{u})/S^1}(\mathbf{z}) = z T^{u} + \frac{1}{z} T^{1 + n \nu-u}.
$$

We explain how to compute the disk potential of $W_{L(\mathbf{u})}$ from $W_{L(\mathbf{u})/S^1}$. To apply Corollary~\ref{cor_countinginv}, we first verify that 
$$
\mathcal{M}_1(P,L(\mathbf{u}),J,\beta) = \overline{\mathcal{M}}_1(P,L(\mathbf{u}),J,\beta).
$$ 
Since $L(\mathbf{u})/S^1$ is positive, $L(\mathbf{u})$ is also positive by Lemma~\ref{lemma_positivity}. Let $\beta \in \Pi^{(2)}(P, L(\mathbf{u}))$. Since $q_*\beta \in \Pi^{(2)}(\CP^1, L(\mathbf{u})/S^1)$ is regular, Lemma~\ref{lemma_qregular} shows that $\beta \in \Pi^{(2)}(P, L(\mathbf{u}))$ is also regular. Hence, the moduli space $\mathcal{M}_1(P,L(\mathbf{u}),J,\beta)$ is a smooth manifold without boundary. 

A priori, however, $\mathcal{M}_1(P,L(\mathbf{u}),J,\beta)$ need not be compact since a sequence of holomorphic disks may escape to infinity. Nevertheless, by Corollary~\ref{corollary_ontophi} and Lemma~\ref{lemma_onetoonephi}, 
$$
\phi \colon \mathcal{M}_1(P,L(\mathbf{u}),J,\beta) /S^1 \to \mathcal{M}_1(\CP^1, L(\mathbf{u})/S^1, q_* J, q_* \beta)
$$ 
is a smooth bijection between smooth manifolds with the same dimension. By the invariance of domain, $\phi$ is a homeomorphism. Since the moduli space on the right-hand side is compact, it follows that $\mathcal{M}_1(P,L(\mathbf{u}),J,\beta) $ is also compact. By Corollary~\ref{cor_countinginv}, the corresponding counting invariants agree.

The remaining task is to understand how the Laurent monomials change under the lifting procedure. Let $\beta$ be the relative homotopy class of a holomorphic disk corresponding to the first term. Then the second term corresponds to the class $[S^2] - \beta$. Via the isomorphism $q_*$ in Lemma~\ref{lemma_isoonhomotopy}, these classes lift to relative classes $\beta_1, \beta_2 \in \pi_2(P, L(\mathbf{u}))$. However, a spherical class in the reduced space does not necessarily lift to a spherical class in $P$. In particular,
$$
\partial \beta_1 + \partial \beta_2 \neq 0.
$$ 

To measure this discrepancy, consider the associated bundle defined by the standard representation, which again yields $\mathcal{O}(-n)$. Let $\beta_0 \in \pi_2(\mcal{O}(-n), L(\mathbf{u}))$ be the class represented by a fiberwise holomorphic disk intersecting the zero section $o_{\CP^1}$ once. By adding a suitable multiple of $\beta_0$ so that the boundary vanishes, we obtain a spherical class
$$
\alpha \coloneqq \beta_1 + \beta_2 + a \beta_0 \in \pi_2(\mathcal{O}(-n)).
$$
Since $\alpha$ is homotopic to the zero section $o_{\CP^1}$ and $\langle c_1(\mathcal{O}(-n)), \alpha \rangle = -n$, it follows that $a = - n$. Consequently,
$$
\partial \beta_2 = - \partial \beta_1 + n \partial \beta_0
$$
and hence the disk potential of $L(\mathbf{u})$ in the total space $P$ is
$$
W_{L(\mathbf{u})}(\mathbf{y}) = y_1 T^u + \frac{y_2^n}{y_1 } T^{1 + n \nu-u}.
$$
\end{example} 

\begin{remark}
As another direction of generalization, consider the base affine space $G/U$ where $G = \mathrm{SL}_2(\C)$ and $U$ is the unipotent radical of a Borel subgroup $B$. The natural projection $\pi \colon G/U \to G/B$ is isomorphic to $\mcal{O}(-1) \backslash o_{\CP^1} \to \CP^1$. In particular, the disk potential computed via our lifting procedure agrees with the Gross--Hacking--Keel--Kontsevich (GHKK) superpotential. which is expressed as a sum of $\theta$-functions. In forthcoming work \cite{Kim}, we will develop this lifting procedure systematically and recover the GHKK superpotential for the basic affine spaces of $\mathrm{SL}_n(\C)$ through holomorphic disk counting.
\end{remark}

\section{Disk potential of quadric hypersurfaces}

In \cite{Kim23}, the disk potential of a monotone Lagrangian torus fiber in the Gelfand–Zeitlin (GZ) system of a quadric hypersurface was computed. In this section, we extend \cite[Theorem 4.1]{Kim23} to arbitrary torus fibers, without assuming monotonicity, by applying the techniques developed in this paper together with a degeneration method.

We begin by fixing notation. Let $\mathbf{x} \coloneqq [x_0 : \dots : x_{n+1}]$ be the homogeneous coordinates on $\CP^{n+1}$ and consider the quadric hypersurface
$$
\mathcal{Q}^{n} \coloneqq \left\{ \mathbf{x} \in \CP^{n+1} \mid x_0^2 + x_1^2 + \dots + x^2_{n+1} = 0 \right\},
$$
equipped with the K\"{a}hler form induced by a suitable multiple of the Fubini–Study form on $\CP^{n+1}$.

The quadric $\mathcal{Q}^n$ carries a completely integrable system, called the \emph{Gelfand–Zeitlin system}, constructed in \cite{Thi81, GS83}. Setting $\| \mathbf{x} \|^2 = \sum_{j=0}^{n+1} |x_j|^2$, we define
\begin{itemize}
\item $\displaystyle \Phi_1(\mathbf{x}) \coloneqq \frac{n \sqrt{-1} (x_0 \overline{x}_1 - \overline{x}_0 x_1)}{\|\mathbf{x}\|^2}$,
\item $\displaystyle \Phi_k(\mathbf{x}) \coloneqq \sqrt{ \left( \frac{n \sum_{i=0}^{k}|x_i|^2}{\|\mathbf{x}\|^2}\right)^2 - \left| \frac{n \sum_{i=0}^{k} x_i^2}{\|\mathbf{x}\|^2}\right|^2} \quad \mbox{for $k=2, \dots, n$}.$ 
\end{itemize}
The GZ system is given by
$$
\Phi_\mathrm{GZ} = (\Phi_1, \dots, \Phi_n) \colon \mathcal{Q}^n \to \R^n.
$$
Its image is the polytope
$$
\Delta^n \coloneqq \left\{ \mathbf{u} \in \R^n \mid n \geq u_n \geq u_{n-1} \geq \dots \geq u_2 \geq |u_1| \right\},
$$
which is a an $n$-simplex. Setting $u_{n+1} \coloneqq n$, we define
$$
\begin{cases}
f_0 \coloneqq \{ \mathbf{u} \in \Delta^n \mid u_2 + u_1 = 0 \} \\
f_j \coloneqq \{ \mathbf{u} \in \Delta^n \mid u_{j+1} - u_j = 0 \} \quad \mbox{for $j = 1, \dots, n$.}
\end{cases}
$$
Then $f_0, f_1, \dots, f_n$ are precisely the facets of $\Delta^n$.

We now recall the toric degeneration of the GZ system constructed by Nishinou--Nohara--Ueda \cite{NNU08, NNU10}. Consider the following degeneration in stages$\colon$
$$
\mathcal{X} \coloneqq \{ (\mathbf{x}, \textbf{t}) \in \CP^{n+1} \times \C^{n-1} \mid x_0^2 + x_1^2 + x_2^2 + t_3 x_3^2 + \dots + t_{n+1} x^2_{n+1} = 0\},
$$
where $\textbf{t} = (t_3, \dots, t_{n+1})$. Let $\pi \colon \mathcal{X} \to \C^{n-1}$ be the projection and we denote by $\mathcal{X}_\mathbf{t} \coloneqq \pi^{-1}(\mathbf{t})$ the fiber over $\mathbf{t}$. 

For each $k = 3, 4, \dots, n+1$, define 
$$
\mathcal{X}_{(k,t_k)} \coloneqq \pi^{-1}(1, \dots, 1, t_k, 0, \dots, 0) = \left\{ \mathbf{x}  \in \CP^{n+1} \mid x_0^2 + x_1^2 + \dots + x_{k-1}^2 + t_{k} x^2_{k} = 0 \right\}.
$$
We also introduce functions $\Psi_1, \dots, \Psi_n$ by
\begin{itemize}
\item $\displaystyle \Psi_1 (\mathbf{x}) \coloneqq \Phi_1 (\mathbf{x}) = \frac{n \sqrt{-1} (x_0 \overline{x}_1 - \overline{x}_0 x_1)}{\|\mathbf{x}\|^2}$,
\item $\displaystyle \Psi_{k} (\mathbf{x}) \coloneqq \frac{n \sum_{i=0}^{k}|x_i|^2}{\|\mathbf{x}\|^2} \quad \mbox{for $k=2, \dots, n$}.$
\end{itemize}
Each $\mathcal{X}_{(k,t_k)}$ carries a completely integrable system 
\begin{equation}\label{equ_phiktk}
\Phi_{(k, t_k)} \coloneqq (\Phi_1, \dots, \Phi_{k-1}, \Psi_{k}, \dots, \Psi_{n} ) \colon \mathcal{X}_{(k, t_k)} \to \R^n. 
\end{equation}
By construction, $\Phi_{(n+1, 1)} = \Phi_\mathrm{GZ}$. On the central fiber $\mathcal{X}_{(3,0)}$, we have $\Phi_2 = \Psi_2$ and hence
$$
\Phi_{(3,0)} = (\Psi_1, \dots, \Psi_n) \eqqcolon \Psi_\mathrm{GZ}
$$ 
which coincides with the toric moment map on 
$$
\mathcal{T}^n \coloneqq  \left\{ \mathbf{x} \in \CP^{n+1} \mid x_0^2 + x_1^2 + x_2^2 = 0 \right\}.
$$
Both $\Phi_{\mathrm{GZ}}$ and $\Psi_{\mathrm{GZ}}$ have image $\Delta^n$. 

To interpolate between these systems, consider a piecewise linear path $\gamma \colon [0,1] \to \C^{n-1}$ defined by
\begin{equation}\label{equ_defgammapie}
(1, \dots, 1) \to (1, \dots, 1, \delta) \to (\delta^{n-1}, \delta^{n-2}, \dots, \delta) \to (\delta^{n-1}, \dots, \delta^{n-1}) \to (0, \dots, 0)
\end{equation}
where $\delta>0$ is sufficiently small. This path approximates the degeneration in stages described in~\eqref{equ_phiktk}. Using the gradient Hamiltonian flow, we transport the GZ system along $\gamma$ to obtain a completely integrable system over the point $(\delta^{n-1}, \dots, \delta^{n-1})$ . Taking the limit as $\delta \to 0$, we obtain a family of completely integrable systems along the line segment $(\delta^{n-1}, \dots, \delta^{n-1}) \to (0, \dots, 0)$, which converges to the toric moment map $\Psi_{\mathrm{GZ}}$.

We write ${X}_t \coloneqq X_{\gamma(t)}$. Then $X_t \simeq \mathcal{Q}^n$ for $t \neq 0$ and $X_0 = \mathcal{T}^n$. We denote by $\Phi_t \colon {X}_t \to \Delta^n$ a completely integrable system on ${X}_t$. Let $\mathring{\mathcal{T}}^n$ be the smooth locus of $\mathcal{T}^n$. 

\begin{theorem}[\cite{NNU08}]\label{theorem_NNU08}
For each $t \in [0,1]$, there exists a continuous map $\psi_t \colon {X}_t \to {X}_0$ such that $\psi_t$ restricts to a symplectomorphism on $\Phi^{-1}_t(\Psi_\mathrm{GZ}(\mathring{\mathcal{T}^n}))$ and the following diagram commutes. 
\begin{equation}\label{equ_toricdegenerations}
\xymatrix{
{X}_t \ar^{\psi_t}[rr] \ar[rd]_{\Phi_t} & & X_0 = \mathcal{T}^n \ar[ld]^{\Phi_0 = \Psi_\mathrm{GZ}} \\
 & \Delta^n &}
\end{equation} 
\end{theorem}

For $\mathbf{u}\in\mathrm{Int}(\Delta^n)$, the fiber $\Phi_t^{-1}(\mathbf{u})$ is diffeomorphic to the toric fiber $\Psi_\mathrm{GZ}^{-1}(\mathbf{u})$ via $\psi_t$. By abuse of notation, we denote this fiber by  $L(\mathbf{u}) \coloneqq \Phi_t^{-1}(\mathbf{u})$. Let $\theta_j = \theta_j(\mathbf{u})$ denote the oriented $S^1$-cycle generated by the action variable $u_j$. These cycles form a basis 
\begin{equation}\label{equ_basisH1}
H_1(L(\mathbf{u}); \Z) = \mathbb{Z} \langle \theta_1, \dots, \theta_n \rangle \simeq \Z^n.
\end{equation}
Let $\{ e_1, \dots, e_n \}$ be the dual basis of $\{ \theta_1, \dots, \theta_n\}$ and let $y_j$ be the exponential variable corresponding to $e_j$. 

Let $\Delta^{\vee}$ be the Newton polytope dual to $\Delta^n$, the convex hull of the primitive inward normal vectors to the facets of $\Delta^n$. The polytope $\Delta^{\vee}$ consists of $(n+1)$ vertices given by
\begin{equation}\label{equ_verticesdeln}
(-1,1,0, \dots 0), (1,1,0, \dots 0), (0,-1,1,0, \dots 0), \dots,  (0, \dots,0, -1, 1), (0, \dots,0, -1) \in \Z^n
\end{equation}
and contains $(n+3)$ lattice points. By the classification result of Cho--Poddar \cite{ChoPoddar}, a holomorphic disk of Maslov index two bounded by a toric fiber in $\mathcal{T}^n$ correspond to a nonzero lattice point of $\Delta^\vee$. In particular, there are no nonconstant holomorphic disks whose boundary class is trivial in $H_1(L(\mathbf{u}); \Z)$. 

We now determine the coefficient of the monomial corresponding to each vertex in~\eqref{equ_verticesdeln}. Using the toric degeneration~\eqref{equ_toricdegenerations}, one can determine most terms of the disk potential of $L(\mathbf{u})$ in $X_t$ for sufficiently small $t$. Each vertex of $\Delta^\vee$ corresponds to a holomorphic disk intersecting a toric divisor and the corresponding open Gromov–Witten invariant is equal to one, see \cite{ChoOh, ChoPoddar}. We refer to these as \emph{basic disk classes}. Moreover, every basic disk class is regular with respect to the standard complex structure. By openness of regularity, they remain regular for sufficiently small $t$ with respect to the complex structure on $X_t$ and the corresponding counting invariants are preserved.

It remains to determine whether there exists an additional contribution from holomorphic disks corresponding to the lattice point $(0,1,0,\dots,0)$. If such disks exist, then the disk potential $W$ contains a term of the form $y_2$, whose coefficient must be determined. Consequently, in this case, the disk potential of $L(\mathbf{u})$ in $X_t$ is of the following form
\begin{equation}\label{equ_diskpotentialfunguess}
W_{L(\mathbf{u})}(\mathbf{y}) = \frac{1}{y_n} T^{n-u_n} + \frac{y_n}{y_{n-1}}  T^{u_n-u_{n-1}}  + \cdots + \frac{y_2}{y_{1}}  T^{u_1 - u_2} + m y_2 T^{u_2} +  y_1 y_2 T^{u_1 + u_2}
\end{equation}
for some coefficient $m$.

For this purpose, we establish the following lemma.

\begin{lemma}\label{lemma_nonintersection}
Assume that $n \geq 3$. Let $\beta$ be a relative homotopy class realized by a Maslov index two holomorphic disk whose boundary class corresponds to the lattice point $(0,1,0, \dots ,0)$. Then, for all sufficiently small $t$, every holomorphic disk $\varphi \in \mathcal{M}(X_t, L(\mathbf{u}), J, \beta)$ satisfies
$$
\varphi(\mathbb{D}) \cap 
\left( \bigl(V(x_0) \cap \dots \cap V(x_3) \bigr) \cup V(x_4 x_5 \dots x_{n+1}) \right) = \emptyset.
$$
\end{lemma}

\begin{proof}
Let $\beta_0, \beta_1, \dots, \beta_n$ be the basic disk classes corresponding to the vertices in~\eqref{equ_verticesdeln}, respectively. Recall that for $n \geq 3$, the induced homomorphism 
$$
\psi_{t,*} \colon H_2(X_t, L(\mathbf{u}) ; \Q) \to H_2(X_0, L(\mathbf{u}); \Q)
$$ 
is an isomorphism that preserves the Maslov index and the symplectic area, see \cite[Lemmas 5.5, 5.6]{Kim23}. By abuse of notation, we continue to denote by $\beta_j$ the classes $\psi_{t,*}^{-1}(\beta_j)$.

Since $\{ \beta_0, \dots, \beta_n \}$ forms a basis for $H_2(X_t, L(\mathbf{u}) ; \Q)$ for $n \geq 3$, we may write
$$
\beta = \sum_{j=0}^{n} c_j \beta_j \quad \mbox{ $c_j \in \mathbb{Q}$}.
$$
The conditions that $\partial \beta$ corresponds to the lattice point $(0,1,0, \dots, 0)$ via~\eqref{equ_basisH1} and the Maslov index is equal to two yield that
$$
c_0 = c_1,\,\, c_2 = \dots = c_n, \,\, c_0 + c_1 - c_2 = 1 , \mbox{ and } \sum_{j=0}^n c_j = 1.
$$ 
It follows that $c_0 = c_1 = 1/2$ and $c_3 = \dots = c_n = 0$ and hence $\beta = (\beta_0 + \beta_1)/2$.

Consider a sequence of holomorphic disks in $X_t$ representing $\beta$ as $t \to 0$. By Gromov compactness theorem, this sequence converges to a stable bordered holomorphic map in the compactification of $\mathcal{M}(X_0, L(\mathbf{u}), J, \beta)$. By the classification of Cho–Poddar \cite{ChoPoddar}, any such disk has minimal symplectic area among classes with boundary $(0,1,0,\dots,0)$ and hence sphere bubbling cannot occur as no excess energy is available. Moreover, when $t = 0$, every holomorphic disk whose boundary corresponds to $(0,1,0,\dots,0)$ avoids $\left( \bigl(V(x_0) \cap \dots \cap V(x_3) \bigr) \cup V(x_4 x_5 \dots x_{n+1})  \right)$. This non-intersection property persists for sufficiently small $t$. This completes the proof.
\end{proof}

Our strategy to compute the coefficient $m$ in~\eqref{equ_diskpotentialfunguess} is to reduce the problem to lower-dimensional quadrics via a GIT quotient while preserving the relevant counting invariants. For this purpose, we consider a ``partial" degeneration of $\mathcal{Q}^n$ whose GIT quotient by a suitable torus action is isomorphic to $\mathcal{Q}^2$. Using the correspondence theorem for GIT quotients, we show that the counting invariant agrees with that of the lower-dimensional quadrics. This reduction enables us to compute the invariants in higher dimensions.

\begin{theorem}[Theorem~\ref{thmx_C}]\label{thm_quadricshypdisk}
Let $n \geq 2$ and let $L^n(\mathbf{u})$ be the Lagrangian torus over $\mathbf{u} \in \mathrm{Int}(\Delta^n)$ in $\mathcal{Q}^n$. Then the disk potential is 
\begin{equation}\label{equ_WLn}
W_{L^n(\mathbf{u})}(\mathbf{y}) = 
\frac{1}{y_n} T^{n-u_n} + \frac{y_n}{y_{n-1}}  T^{u_n-u_{n-1}}  + \dots + \frac{y_2}{y_{1}}  T^{u_1 - u_2} + 2y_2  T^{u_2} +  y_1 y_2 T^{u_1 + u_2}.
\end{equation}
\end{theorem}

\begin{proof} 
When $n = 2$, it is known (see \cite{AurouxTduality, FOOOS2S2} for instance) that the disk potential $L^{2}(\mathbf{u})$ in $\mathcal{Q}^{2} \simeq \CP^1 \times \CP^1$  is 
\begin{equation}\label{equ_diskq2}
W_{L^2(\mathbf{u})} (\mathbf{y}) = \frac{1}{y_2} T^{2-u_2}+ \frac{y_2}{y_{1}} T^{u_2 - u_1} + 2y_2 T^{u_2}  + y_1 y_2 T^{u_1 + u_2}.
\end{equation}

Assume that $n \geq 3$. Let $G \coloneqq (S^1)^{n-2}$ and $\mathbb{G} \coloneqq (\C^*)^{n-2}$. To prove~\eqref{equ_WLn} for $\mathcal{Q}^n$, we reduce the computation to the case $n=2$ via a GIT quotient. Consider the $G$-action on $\CP^{n+1}$ given by
\begin{equation}\label{equ_Sn3action}
G \times \CP^{n+1} \to \CP^{n+1}, \quad  \mathbf{s} * \mathbf{x} = [x_0 : \dots : x_3 : s_4 x_4: \dots : s_{n+1} x_{n+1}]
\end{equation}
where $\mathbf{s} \coloneqq (s_4, \dots, s_{n+1}) \in G$. After a change of coordinates, the associated moment map can be taken to be $\mu_{G} (\mathbf{x}) = (\Psi_3(\mathbf{x}), \dots, \Psi_n(\mathbf{x}))$. The $G$-action is free on the level set $\mu_{G}^{-1}(2, \dots , n-1)$. 

We extend this action to a $\mathbb{G}$-action and consider the linearization on $\mathcal{L} = \mathcal{O}(n)$ determined by the character $\chi$ corresponding to $(2, \dots , n-1)$. The resulting GIT quotient is 
$$
\CP^{n+1}  \git_{\chi} (\C^*)^{n-2} \simeq \CP^3.
$$ 
The stable locus is 
$$
\CP^{n+1,s} =\CP^{n+1} \setminus \left( \bigl(V(x_0) \cap \dots \cap V(x_3) \bigr) \cup V(x_4 x_5 \dots x_{n+1}) \right)
$$
and the projection $q$  
\begin{equation}\label{equ_themorphismq}
q \colon \CP^{n+1,s} \to \CP^{3}, \quad  [x_0 \colon \dots \colon x_{n+1} ] \mapsto [x_0 \colon \dots \colon x_3]
\end{equation}
is a principal $\mathbb{G}$-bundle.

Since the quadric $\mathcal{Q}^n$ is \emph{not} invariant under the action~\eqref{equ_Sn3action}, we consider the following degeneration. Consider a piecewise linear path $\gamma^\prime \colon [0,1] \to \C^{n-1}$ (cf.~\eqref{equ_defgammapie}) as 
$$
(1, \dots, 1) \to (1, \dots, 1, \delta) \to (t_3, \delta^{n-2}, \dots, \delta) \to (t_3, \delta^{n-2}, \dots, \delta^{n-2}) \to (t_3, 0, \dots, 0)
$$ 
Using the gradient Hamiltonian flow, we transport the GZ system along $\gamma^\prime$ to obtain a completely integrable system over $(t_3, \delta^{n-2}, \dots, \delta^{n-2})$. Taking $\delta \to 0$, we obtain a family of completely integrable system on each point of the line segment $(t_3, \delta^{n-2}, \dots, \delta^{n-2})$, which converges to $\Phi_{(3, t_3)}$ in $\mathcal{X}_{(3, t_3)}$.

The reason for considering such a degeneration (rather than the toric degeneration in~\eqref{equ_toricdegenerations}) is to facilitate comparison of the holomorphic disks corresponding to $(0,1,0,\dots,0)$. We denote $X^\prime_t \coloneqq \mathcal{X}_{\gamma^\prime(t)}$. As in Theorem~\ref{theorem_NNU08}, there exists a commutative diagram
\begin{equation}\label{equ_toricdegenerations2}
\xymatrix{
{X}^\prime_t \ar^{\psi^\prime_t}[rr] \ar[rd]_{\Phi^\prime_t} & &X^\prime_0 = \mathcal{X}_{(3, t_3)}\ar[ld]^{\Phi_{(3, t_3)}} \\
 & \Delta^n &}
\end{equation} 

Let $\beta$ be a relative homotopy class realized by a Maslov index two holomorphic disk whose boundary  class corresponds to the lattice point $(0,1,0, \dots ,0)$. By Lemma~\ref{lemma_nonintersection}, for all sufficiently small $t_3$, every holomorphic disk in $\beta$ (if exists) is contained in the (semi)stable locus. In particular, the class $\beta$ is semistable in $X_{(t_3, 0)}$. 

Choose a point $\mathbf{u} =  (u_1, u_2, u_3 = 2, \dots, u_n = n-1) \in \mathrm{Int}(\Delta^n)$ lying in the level set $\mu_{G}^{-1}(2, \dots, n-1)$ and set $\mathbf{v} = (u_1, u_2)$. Let $L^n(\mathbf{u})$ denote the fiber of $\Phi_{(n, 0)}$ over $\mathbf{u}$ in $\mathcal{X}_{(3,t_3)}$. By~\eqref{equ_phiktk}, the quotient map sends $L^n(\mathbf{u})$ to $L^{2}(\mathbf{v}) \subseteq \mathcal{Q}^{2}$ . Recall from~\eqref{equ_basisH1} that
$$
H_1(L_n(\mathbf{u}); \Z) \simeq \Z \langle \theta_1, \dots, \theta_n  \rangle.
$$ 
Moreover, $q_* \colon \Z \langle \theta_1, \theta_{2} \rangle \to H_1(L_{2}(\mathbf{v});\Z)$ is an isomorphism and $\{ \theta_3, \dots, \theta_n \}$ generates the kernel of $q_*$. Consequently, $y_1$ and $y_2$ survive and all other $y_\bullet$ becomes $1$ under the quotient map $q_*$.

Recall that there are two homotopy classes ${\alpha}$ and $\alpha^\prime$ of Maslov index two whose boundary corresponds to $(0,1) \in \Z^2$. Let $\widetilde{\alpha}$ and $\widetilde{\alpha}^\prime$ are unique homotopy classes in $\pi_2(\mathcal{X}_{(3, t_3)}, L(\mathbf{u}))$ such that $q_*(\widetilde{\alpha}) = \alpha$ and $q_*(\widetilde{\alpha}^\prime) = \alpha^\prime$. By Lemma~\ref{lemma_qregular} and Corollary~\ref{cor_nbetanalpha}, $\widetilde{\alpha}$ and $\widetilde{\alpha}^\prime$ are also regular and the sum of their counting invariants is equal to $2$, that is,
$$
n_{\widetilde{\alpha}} + n_{\widetilde{\alpha}^\prime} = n_\alpha + n_{\alpha^\prime} = 2.
$$
Moreover, $\partial  \widetilde{\alpha}$ corresponds to $(0,1,0,\dots, 0)$ because the disk potential must have the form~\eqref{equ_diskpotentialfunguess}. By Lemma~\ref{lemma_nonintersection}, every holomorphic disk whose boundary class corresponding to $(0,1,0,\dots, 0)$ lies in the semistable locus. Therefore,
\begin{equation}\label{equ_unionalphaalphaprime}
\mathcal{M}_1(X^\prime_0, L(\mathbf{u}), J, \iota_* \widetilde{\alpha}) \cup \mathcal{M}_1(X^\prime_0, L(\mathbf{u}), J, \iota_* \widetilde{\alpha}^\prime)
\end{equation}
are the only moduli spaces of holomorphic disks of Maslov index two that contribute to $y_2T^{u_2}$. Hence, its coefficient is exactly two. Since the moduli spaces in~\eqref{equ_unionalphaalphaprime} are regular, the corresponding counting invariant remains the same for sufficiently small $t$. Hence, the disk potential of $L(\mathbf{u})$ in $X^\prime_t$ is of the form~\eqref{equ_WLn}.

We have determined the disk potential of $L(\mathbf{u})$ for $\mathbf{u} \in \mu_{G}^{-1}(2, \dots, n-1)$. It remains to compute the disk potential of $L(\mathbf{u})$ with arbitrary $\mathbf{u} \in \mathrm{Int}(\Delta_\mathrm{GZ})$. For $\varphi$ in~\eqref{equ_unionalphaalphaprime}, consider the composition
$$
\eta \circ \varphi \colon (\mathbb{D}, \partial \mathbb{D}) \to (X^\prime_0, L(\mathbf{u})) \to ((\C^*)^{n-2}, (S^1)^{n-2})
$$
where
$$
\eta( [x_0 : x_1 : \dots : x_{n+1}] ) \mapsto \left(\frac{x_3}{x_2}, \dots, \frac{x_{n+1}}{x_2} \right).
$$
The variety $X_0^\prime$ carries a Hamiltonian $T^{n-2}$-action generated by $(\Psi_4 - \Psi_3, \dots, \Psi_n - \Psi_{n-1}, - \Psi_n)$ and $(\C^*)^{n-2}$ carries the standard torus action.  The map $\eta$ is $T^{n-2}$-equivariant with respect to these actions. By Lemma~\ref{lemma_nonintersection}, $\varphi$ has neither zeros nor poles along the divisor $V(x_2 \dots x_{n+1})$. Thus $\eta \circ \varphi$ is a well-defined holomorphic map and it must be constant by the maximum modulus principle. Consequently, we obtain an orientation-preserving diffeomorphism
$$
\mathcal{M}_1(X^\prime_0, L(\mathbf{u}), J, \beta) \to \mathcal{M}_1(\mathcal{Q}^{2}, L(\mathbf{v}), q_* J, q_* \beta) \times T^{n-2}, \quad [(\varphi, z_0)] \mapsto \bigl([(q \circ \varphi,z_0)], \eta \circ \varphi \bigr)).
$$ 

Now consider another level $(\nu_3, \ldots, \nu_n)$ with $\nu_3 \le \cdots \le \nu_n$ and choose
$$
\mathbf{u}^\prime = (u_1, u_2, \nu_3, \cdots, \nu_{n}) \in \mu_G^{-1}(\nu_3, \cdots, \nu_{n}).
$$
By the same argument, we obtain an isomorphism
$$
\mathcal{M}_1(X^\prime_0, L(\mathbf{u}^\prime), J, \beta) \simeq \mathcal{M}_1(\mathcal{Q}^{2}, L(\mathbf{v}), q_* J, q_* \beta) \times \widetilde{T}^{n-2}, \quad [(\varphi, z_0)] \mapsto \bigl([(q \circ \varphi,z_0)], \eta \circ \varphi \bigr)).
$$
Let $\mathbf{u} = (u_1, u_2, 2, \dots, n-1)$. By taking the product of cylinders from $T^{n-2}$ to $\widetilde{T}^{n-2}$ inside $(\C^*)^{n-2}$, we obtain an orientation-preserving cobordism between $\mathcal{M}_1(X^\prime_0, L(\mathbf{u}), J, \beta)$ and $\mathcal{M}_1(X^\prime_0, L(\mathbf{u}^\prime), J, \beta)$. In particular, the corresponding counting invariants agree. Therefore, the disk potential of $L(\mathbf{u}^\prime)$ is also given by~\eqref{equ_WLn}, as claimed.
\end{proof}

\begin{remark}
Although $\pi_2(\mathcal{Q}^3, {L}(\mathbf{u})) \simeq \pi_2(X^\prime_0, L(\mathbf{u}))$ are isomorphic, there is a birth-and-death phenomenon of homotopy classes under compactification or removal of unstable locus.

On one hand, $\pi_2(\mathcal{Q}^3, {L}(\mathbf{u})) \simeq \Z^{4}$ is generated by the basic disk classes. Among them, the basic disk class corresponding to the term $1/y_3 T^{3-u_3}$ is unstable. The summand generated by this class disappears after passing to the semistable locus. 

On the other hand, $\pi_2(X^\prime_0, L(\mathbf{u})) \simeq \pi_2(\mathcal{Q}^2, {L}(\mathbf{v})) \simeq \Z^{4}$. The disk classes ${\alpha}$ and ${\alpha}^\prime$ define distinct elements in $\pi_2(\mathcal{Q}^2, {L}(\mathbf{v}))$ and similarly $\widetilde{\alpha}$ and $\widetilde{\alpha}^\prime$ define distinct elements in $\pi_2(X^\prime_0, L(\mathbf{u}))$. The classes ${\alpha}$ and ${\alpha}^\prime$ differ by the class generated by a Lagrangian $2$-sphere. However, after compactification, $\widetilde{\alpha}$ and $\widetilde{\alpha}^\prime$ become same. The reason is that the corresponding spherical class can be contracted through the Lagrangian $S^3$-fiber, cf. \cite[Example 7.7]{CK20}.
\end{remark}

\providecommand{\bysame}{\leavevmode\hbox to3em{\hrulefill}\thinspace}
\providecommand{\MR}{\relax\ifhmode\unskip\space\fi MR }
\providecommand{\MRhref}[2]{%
  \href{http://www.ams.org/mathscinet-getitem?mr=#1}{#2}
}
\providecommand{\href}[2]{#2}

\end{document}